\RequirePackage{fix-cm}
\documentclass[smallextended]{svjour3}       
\smartqed  
\usepackage{graphicx}
\usepackage{algorithm}
\usepackage{algpseudocode}
\usepackage{paralist,pifont}

\usepackage{geometry}
\usepackage{latexsym}
\usepackage{amsmath}
\usepackage{amssymb}
\usepackage{times}
\usepackage{booktabs}

\usepackage{color}
\usepackage{subfigure}
\usepackage{epstopdf}
\usepackage{nicefrac}
\numberwithin{equation}{section}

\newcommand{\Dh}{\Delta_h}

\newcommand{\tc}{\tilde{c}}

\newcommand{\bD}{\boldsymbol{D}}
\newcommand{\bJ}{\boldsymbol{J}}

\newcommand{\reff}[1]{{\rm (\ref{#1})}}
\newcommand{\hf}{\frac{1}{2}}

\newcommand{\ciptwo}[2]{\left( #1 , #2 \right)}

\newcommand\dt {{\Delta t}}

\begin{document}

\title{A structure-preserving implicit exponential time differencing scheme for Maxwell-Amp\`ere Nernst-Planck model
}


\author{Yunzhuo Guo \and Qian Yin \and Zhengru Zhang
}


\institute{Yunzhuo Guo \at
              School of Mathematical Sciences, Beijing Normal University, Beijing 100875, P.R. China \\
              \email{yunzguo@mail.bnu.edu.cn}        
           \and
           Qian Yin \at
              {Department of Applied Mathematics, The Hong Kong Polytechnic University, Hung Hom, Hong Kong} \\
               \email{sjtu\_yinq@sjtu.edu.cn}        
           \and
           Zhengru Zhang (Corresponding Author) \at
           Laboratory of Mathematics and Complex Systems, Ministry of Education and School of Mathematical Sciences, Beijing Normal University, Beijing 100875, P.R. China \\
            \email{zrzhang@bnu.edu.cn}        
           \and
}

\date{Received: date / Accepted: date}

\maketitle

\begin{abstract}
The transport of charged particles, which can be described by the Maxwell-Amp\`ere Nernst-Planck (MANP) framework, is essential in various applications including ion channels and semiconductors. We propose a decoupled structure-preserving numerical scheme for the MANP model in this work. The Nernst-Planck equations are treated by the implicit exponential time differencing method associated with the Slotboom transform to preserve the positivity of the concentrations. 
In order to be effective with the Fast Fourier Transform, additional diffusive terms are introduced into Nernst-Planck equations. Meanwhile, the correction is introduced in the Maxwell-Amp\`ere equation to fulfill Gauss's law.
The curl-free condition for electric displacement is realized by a local curl-free relaxation algorithm whose complexity is $O(N)$. 
We present sufficient restrictions on the time and spatial steps to satisfy the positivity and energy dissipation law at a discrete level. 
Numerical experiments are conducted to validate the expected numerical accuracy and demonstrate the structure-preserving properties of the proposed method.

\bigskip

\end{abstract}

\section{Introduction}
The classical Poisson-Nernst-Planck (PNP) equations have many crucial applications in areas of semiconductors~\cite{2012semiconductor}, electrochemical systems~\cite{BTA:PRE:04}, ion channels~\cite{E:CP:98} and so on because of its success in describing charge transport. Based on the mean-field approximation, the PNP equations consist of Nernst-Planck equations for the diffusion and convection of the ionic concentrations and Poisson's equation for the electric potential. In this work, we consider a variant of the original PNP equations, named the Maxwell-Amp\`ere Nernst-Planck (MANP) model \cite{Qiao_MANP_model_2023}, which can be equivalently derived from the PNP equations by replacing the potential with electric displacement. The new model is governed by the following equations~\cite{Qiao_MANP_model_2023}:
\begin{align}
&\frac{\partial c^{\ell}}{\partial t}=\nabla \cdot \kappa\left(\nabla c^{\ell}-\frac{q^{\ell} c^{\ell} \boldsymbol{D}}{\varepsilon}+c^{\ell} \nabla \mu^{\ell, \mathrm{cr}}\right), \ell=1, \cdots, M, \label{MANP_1}   \\
&\frac{\partial \boldsymbol{D}}{\partial t}=\sum_{\ell=1}^M \frac{q^{\ell}}{2 \kappa}\left(\nabla c^{\ell}-\frac{q^{\ell} c^{\ell} \boldsymbol{D}}{\varepsilon}+c^{\ell} \nabla \mu^{\ell, \mathrm{cr}}\right)+\boldsymbol{\Theta}, \label{MANP_2}\\
&\nabla \cdot \boldsymbol{\Theta}=0, \label{MANP_3} \\
&\nabla \times \frac{\boldsymbol{D}}{\varepsilon}=\mathbf{0},\label{MANP_4}
\end{align}
where $c^{\ell}$ denotes the concentration of charged particles of the $\ell$-th species, $q^\ell$ is the associated ionic valence, $\bD$ is the electric displacement, $\varepsilon$ is the relative dielectric coefficient and $\kappa$ is a dimensionless parameter describing the Debye length. The free energy is given by
\begin{equation}
\mathcal{F}\left[c^1, \ldots, c^M\right]=\int_{\Omega}\left(\kappa^2 \frac{|\bD|^2}{\varepsilon}+\sum_{\ell=1}^M c^{\ell} \log c^{\ell} \right) dx +F^{ \text { cr }} \text {, with } \nabla \cdot 2 \kappa^2 \bD=\rho, \label{original energy}
\end{equation}
where $\rho:=\sum_{\ell=1}^M q^\ell c^\ell + \rho^f$ is the total charge density containing the space-dependent fixed charge $\rho^f$. $F^{ \mathrm{cr}}$ is the excess energy beyond the mean-field approximation. The excess chemical potential can be computed by the variation $\mu^{\ell, \mathrm{cr}}=\delta F^{\text{cr}}/\delta c^\ell$ and here we only consider the steric effects and Born solvation interaction \cite{Wang:PRE:2010,Liu2017PRE,KBA:PRE:2007,Qiao_MANP_model_2023}:
\begin{equation}
\mu^{\ell, \mathrm{cr}}=-\frac{v^{\ell}}{v^0} \log \left(v^0 c^0\right)+\frac{\chi\left(q^{\ell}\right)^2}{a^{\ell}}\left(\frac{1}{\varepsilon}-1\right). \label{excess potential}
\end{equation}
In \eqref{excess potential}, $c^0$ can be regarded as the concentration of solvent molecule, which is defined as 
\begin{equation}
c^0:=\left(1-\sum_{\ell=1}^M v^{\ell} c^{\ell}\right) / v^0, \label{solvent concentration}
\end{equation}
with $v^{\ell}$ and $v^0$ being the volumes of $\ell$-th ion and solvent molecule, respectively and $a^\ell$ being the Born radius for ions of the $\ell$-th species. 

The MANP model originated from the observation that the electric potential in the Nernst-Planck equations solely appears as its gradient, driving the convection of ionic concentrations. By substituting potential with the electric field $\boldsymbol{E}=-\nabla \phi$ or displacement $\bD=\varepsilon \boldsymbol{E}$, the Maxwell-Amp\`ere equation can be derived. Compared to the global potential, the electric displacement used in the new equation is local and retarded  diffusely~\cite{MR:PRL:2002,M:JCP:2002}. It can be computed locally by the Maxwell-Amp\`ere equation with the curl-free relaxation. Moreover, the MANP equations also maintain corresponding physical structures, such as positivity, mass conservation and energy dissipation law. 

There are many studies on developing structure-preserving numerical methods for the PNP and PNP-type equations \cite{Liu_PNP_2014,DGPNP_2016,Wang2016A,Modified_2021,dynamic_2022,Ding_PRE_2020,Zhou_2011,Ding_2013,Ding_2024,Ding_2019_optimal}.
On the one hand, the linear scheme generally has high efficiency but requires additional conditions or modifications to maintain specific physical properties. For example, He et al.~\cite{He_PNP_2019} raised a multi-step first-order linear scheme that unconditionally satisfies positivity-preservation. The longer stencil in time is necessary for the energy dissipation in the sense of a modified energy functional. 
The Slotboom transform is also a widely used method for designing linear numerical schemes \cite{Ding_2013a,Ding_2019,Efficient_2021}. The positivity is preserved based on the properties of an M-matrix. Meanwhile, a condition for the time step is generally needed for energy stability.
On the other hand, the nonlinear scheme maintains the physical structure unconditionally at a very high cost of calculating complexity. Ding et al. \cite{Ding_2020} implicitly treated the potential in the Slotboom transform and obtained the unconditional energy dissipation. Since the energy functional is convex, the fully implicit scheme based on discretization of gradient flow formulation is unconditionally energy stable \cite{LiuWang_PNP_2021,Qian_2021,Qian_2023}. An iteration solver for a nonlinear scheme of the classic PNP equations is proposed in \cite{Liu_Iteration_2022} and the convergence was also proved.
Besides, for the new MANP model, Qiao et al.~\cite{Qiao_NumericalMANP_2023} designed a positivity-preserving scheme using a similar idea to the Slotboom transform. Chang et al.~\cite{Chang_deep_learning} proposed a hybrid method combining the tools of deep learning and traditional approach. The positivity of concentration is also proved.

Recently, the exponential time differencing (ETD) method has been widely used to construct numerically stable schemes with physical structures. Due to the accurate integration of linear terms rather than numerical approximation, the ETD method has excellent accuracy and stability. Many studies have systematically developed the ETD methods \cite{Beylkin_ETD_1998,Cox_ETD_2002,Hochbruck_ETDRK_2006,Exponential_2010}, including several nice reviews of high-order methods based on multi-step approach and Runge-Kutta framework, as well as convergence analysis. 
In practical applications, the dense stiffness matrix generated by the exponential function can be efficiently calculated using the Fast Fourier Transform (FFT). As a result, the ETD method has attracted significant attention in the field of gradient flow problems \cite{Stability_analysis_2004,contour_integration_2005,ETDRK_elastec_bending_2016,Ju_ETD_epitaxial_2017,Li_ConvergenceETD_2019,Yang_Energy_ETD_2022}, especially designing structure-preserving schemes. Zhu et al. introduced a high-order ETD method for a class of parabolic equations \cite{FastHighOrder_2016}. In their work, a two-step compact difference method was used to achieve fourth-order accuracy in space and the Runge-Kutta discretization was used to obtain high-order accuracy in time. Du et al. \cite{Du_MBP_nolocal_2019} gave the proof of the maximum bound principle of the ETD method for the non-local Allen-Cahn equation.

In this work, the ETD framework is established for the Nernst-Planck equations, associated with the local curl-free relaxation algorithm for the Maxwell-Amp\`ere equations. There are two issues. The first one is that the classic ETD framework is basically explicit. As for MANP equations, physical properties cannot be theoretically proved in explicit scheme although numerical calculation shows agreement. Another issue is the mobility. PNP and PNP-type equations belong to the class of gradient flow problems with variable mobility while the ETD only has advantages when dealing with problems with constant ones. This is because the variable mobility can be reflected in the matrix exponential function, making the FFT inapplicable. 
To overcome these issues, we incorporate additional diffusive terms to implement the ETD framework. By combining the implicit treatment of certain linear terms, we can theoretically demonstrate that the numerical method preserves positivity and energy stability under acceptable conditions. 

The rest of this paper is organized as follows. In Section \ref{section discretization in space}, we define the centered finite difference spatial discretization and corresponding inner products and norms for mesh functions. The detailed implicit ETD method and local curl-free relaxation are developed in Section \ref{section numerical method}. Corresponding numerical analysis on structure-preserving properties is given in Section \ref{section structure-preserving properties}. Simulations are carried out in Section \ref{section numerical test} to demonstrate the reliability of the proposed algorithm. Conclusions are made in Section \ref{section conclusion}.

\section{Spatial Discretization} \label{section discretization in space}
The standard centered finite difference spatial discretization is applied to \eqref{MANP_1}-\eqref{MANP_2}. We adopt the notations and results for discrete functions and operators used in \cite{chen_convCHHS_2015,wise_PFC_2009,shen_Second-order_2012,hu_PFC_2009}. With the computational domain~$\Omega=(0, L_{x}) \times (0, L_{y})$~, we present the following definitions and properties.

\subsection{Basic definitions}
Consider that the domain $\Omega$ is covered by a uniform grid with size $h=L_{x}/N_{x}=L_{y}/N_{y}$, where the integers $N_x$ and $N_y$ are grid numbers along each dimension. We introduce the following representations for discrete mesh and function spaces.
\begin{definition}
\begin{compactitem}
	\item[$(1)$] For any positive integer $N$, we denote:
	$$
	\mathcal{E}_{N}:=\{i \cdot h | i=0, \ldots, N \}, \quad \mathcal{C}_{N}:=\{(i-1 / 2) \cdot h | i=1, \ldots, N \},
	$$
	$$
	\mathcal{C}_{\bar{N}}:=\{(i-1 / 2) \cdot h | i=0, \ldots, N+1\}.
	$$
	
	The points belong to $\mathcal{C}_{\bar{N}} \setminus \mathcal{C}_{N} $ are the so-called ghost points.
	
	\item[$(2)$] Define the following function spaces.
	\begin{compactitem}
		\item Scalar cell-centered function:
		$$\mathcal{C}_{\Omega}:=\{\phi: \mathcal{C}_{\bar{N}_{x}} \times \mathcal{C}_{\bar{N}_{y}} \rightarrow \mathbb{R}\},$$
		where $\ \phi_{i,j}:=\phi(\xi_i, \xi_j)$, $\xi_i:=(i-1/2) h$ and $\xi_j:=(j-1/2)h$.
		\item Scalar vertex-centered function:
		$$\mathcal{F}_{\Omega}:=\{\phi: \mathcal{E}_{\bar{N}_{x}} \times \mathcal{E}_{\bar{N}_{y}} \rightarrow \mathbb{R}\}.$$
		
		\item Scalar face-centered function:
		$$
		\mathcal{E}_{\Omega}^{x}:=\{f: \mathcal{E}_{N_{x}} \times \mathcal{C}_{N_{y}} \rightarrow \mathbb{R}\},\quad
		\mathcal{E}_{\Omega}^{y}:=\{f: \mathcal{C}_{N_{x}} \times \mathcal{E}_{N_{y}} \rightarrow \mathbb{R}\},
		$$
		where $f_{i+\hf,j}:= f(\xi_{i+\hf},\xi_j) \in \mathcal{E}_{\Omega}^{x}$ and $f_{i,j+\hf}:= f(\xi_{i},\xi_{j+\hf}) \in \mathcal{E}_{\Omega}^{y}$.
		\item Vector function:
		$$\mathcal{E}_{\Omega}:=\mathcal{E}_{\Omega}^{x} \times \mathcal{E}_{\Omega}^{y}=\left\{\boldsymbol{f}= [f^x, f^y]^T: f^x\in \mathcal{E}_{\Omega}^{x} \text{~and~}  f^y\in \mathcal{E}_{\Omega}^{y}\right\}.$$
	\end{compactitem}
\end{compactitem}
For convenience, we omit the upper index in the function $\boldsymbol{f} = [f^x, f^y]^T \in \mathcal{E}_{\Omega}$ and denote 
\begin{equation}
f_{i+\frac{1}{2},j}=f^x_{i+\frac{1}{2},j}, \quad f_{i,j+\frac{1}{2}}=f^y_{i,j+\frac{1}{2}}. \nonumber
\end{equation}
The discrete boundary condition, associated with cell-centered function, is proposed in Definition \ref{periodic boundary condition}.
\end{definition}

\begin{definition} \label{periodic boundary condition}
 A discrete function $\phi \in \mathcal{C}_{\Omega}$ satisfies periodic boundary condition if
$$
\begin{array}{ll}
{\phi_{0,j}=\phi_{N_x,j},} & {\phi_{N_x+1,j}=\phi_{1,j},} \\
{\phi_{i,0}=\phi_{i,N_y},} & {\phi_{i,N_y+1}=\phi_{i,1}.} 
\end{array}
$$
\end{definition}
\subsection{Discrete operators, inner products, and norms}
Based on central difference, the discrete differential operators are defined as follow.
\begin{definition}
\begin{itemize}
	\item[$(1)$] $d_{x}: \mathcal{E}_{\Omega}^{x} \rightarrow \mathcal{C}_{\Omega}$ is defined component-wisely by
	$$d_{x} f_{i, j}:=\frac{1}{h}\left(f_{i+\frac{1}{2}, j}-f_{i-\frac{1}{2}, j}\right),$$
	and $d_{y}: \mathcal{E}_{\Omega}^{y} \rightarrow \mathcal{C}_{\Omega} $ is formulated analogously. Then we have discrete divergence
	$$\nabla_{h} \cdot: \mathcal{E}_{\Omega} \rightarrow \mathcal{C}_{\Omega}, \quad\quad \nabla_{h} \cdot \boldsymbol{f}:=d_{x} f^{x}+d_{y} f^{y},$$
	and discrete curl in 2-D
	$$\nabla_{h}\times: \mathcal{E}_{\Omega} \rightarrow \mathcal{F}_{\Omega},\quad\quad 
	\nabla_{h} \times \boldsymbol{f}_{i+\frac{1}{2},j+\frac{1}{2}}:=\frac{1}{h}({f}_{i+1,j+\frac{1}{2}}-{f}_{i,j+\frac{1}{2}})
	-\frac{1}{h}({f}_{i+\frac{1}{2},j+1}-{f}_{i+\frac{1}{2},j}),$$
	where $\boldsymbol{f}=(f^x, f^y)^T \in \mathcal{E}_{\Omega}$.
	\item[$(2)$] $D_{x}: \mathcal{C}_{\Omega} \rightarrow \mathcal{E}_{\Omega}^{x}$ is defined component-wisely by
	$$D_{x} \phi_{i+\frac{1}{2}, j}:=\frac{1}{h}\left(\phi_{i+1, j}-\phi_{i, j}\right),$$
	and $D_{y}: \mathcal{C}_{\Omega} \rightarrow \mathcal{E}_{\Omega}^{y}$ is formulated analogously. Then we have discrete gradient defined by
	$$\nabla_{h}: \mathcal{C}_{\Omega} \rightarrow \mathcal{E}_{\Omega},\quad\quad \nabla_{h} \phi:=\left(D_{x} \phi, D_{y} \phi \right)^{T}.$$
	\item[$(3)$] The standard discrete Laplace operator is defined by
	$$\Delta_{h}: \mathcal{C}_{\Omega} \rightarrow \mathcal{C}_{\Omega}, \quad\quad \Delta_{h} \phi:=\nabla_{h} \cdot \nabla_{h} \phi.$$
\end{itemize}
\end{definition}
In this paper, we regard the cell-centered function as a flattened column vector belongs to $R^{(N_x+2) \times (N_y+2)}$ or $R^{N_x \times N_y}$ if the function satisfies the periodic boundary condition. The discrete linear operators, like Laplacian, are naturally considered as a matrix belongs to $R^{[(N_x+2) \times (N_y+2)]^2}$ or $R^{(N_x \times N_y)^2}$. Now we define the discrete inner products and norms.
\begin{definition}
The inner products for cell-centered function space $\mathcal{C}_{\Omega}$ and face-centered function space $\mathcal{E}_{\Omega}^x$, $\mathcal{E}_{\Omega}^y$ are respectively defined by
$$(\phi, \psi):=h^2 \sum_{i=1}^{N_x} \sum_{j=1}^{N_y} \phi_{i,j} \psi_{i,j}, \quad\quad\quad\quad \forall \phi, \psi \in \mathcal{C}_{\Omega},$$
$$[f, g]_{x}:= \frac{1}{2} h^2 \sum_{i=1}^{N_x} \sum_{j=1}^{N_y}(f_{i+\frac{1}{2}, j} g_{i+\frac{1}{2}, j}+f_{i-\frac{1}{2}, j} g_{i-\frac{1}{2}, j}), \qquad \forall f,g \in \mathcal{E}_{\Omega}^x,$$
and 
$$[f, g]_{y}:= \frac{1}{2}h^2 \sum_{i=1}^{N_x} \sum_{j=1}^{N_y}(f_{i, j+\frac{1}{2}} g_{i, j+\frac{1}{2}}+f_{i, j-\frac{1}{2}} g_{i, j-\frac{1}{2}}), \qquad \forall f,g \in \mathcal{E}_{\Omega}^y.$$
The inner product for discrete vector function is defined by
\begin{equation}\label{vec_product}
	(\boldsymbol{f}, \boldsymbol{g}):=[f^{x}, g^{x}]_{x}+[f^{y}, g^{y}]_{y}, \qquad \forall  \boldsymbol{f}=(f^x,f^y)^T\in \mathcal{E}_{\Omega}, \quad \boldsymbol{g}=(g^x,g^y)^T\in \mathcal{E}_{\Omega}.
\end{equation}
For the cell-centered function, if $\phi$ and $\psi$ satisfy the periodic boundary condition, then the summation by parts formula holds
$$(\phi,\Delta_h\psi)=-(\nabla_h\phi,\nabla_h\psi).$$
This formula will be utilized in following analysis.
\end{definition}

\begin{definition}
	For scalar cell-centered function $\phi\in\mathcal{C}_{\Omega}$, the $p$-norms are defined by
	$$\|\phi\|_{p}:= \left\{
	\begin{aligned}
		&(|\phi|^p,1)^{\frac{1}{p}},\quad\quad 1 \leq p < \infty, \\
		&\max _{i, j}|\phi_{i, j}|,\quad\quad   p=\infty.
	\end{aligned} \right.$$

For vector function $\boldsymbol{f}=(f^x,f^y)^T\in\mathcal{E}_{\Omega}$, the p-norm are defined by
\begin{equation}\label{vecnorm}
\left\|\boldsymbol{f} \right\|_{p}:=\left\{
\begin{aligned} & \left(\left[\left|f^x \right|^{p}, 1\right]_{x}+\left[\left|f^y \right|^{p}, 1\right]_{y}\right)^{\frac{1}{p}}, \qquad 1 \leq p < \infty,\\
	&\max_{i,j} \left\{  |f^x_{i+\frac{1}{2}, j}|,|f^y_{i, j+\frac{1}{2}}| \right\},\qquad\  p=\infty.
\end{aligned} \right.\end{equation}
	
Note that the 2-norm for discrete vector function can also be defined by the above inner product~\eqref{vec_product} in the form of
$$\|\boldsymbol{f}\|_2=\sqrt{(\boldsymbol{f},\boldsymbol{f})}, \qquad \forall\boldsymbol{f}\in \mathcal{E}_{\Omega},$$
 which is consistent with the definition~\reff{vecnorm} when $p=2$.
For simplicity of presentation, we denote $\| \cdot \| = \| \cdot \|_2$ in the rest of the paper.
\end{definition}

\section{Numerical method} \label{section numerical method}
\subsection{The implicit ETD method for Nernst-Planck equations}

By the Helmholtz decomposition theorem~\cite{Nedelec2001,Monk2003FiniteEM}, there exists a scalar function $\phi$ such that $\boldsymbol{D}/\varepsilon =-\nabla \phi$ with curl-free condition in \eqref{MANP_4}. Using $\phi$ and the Slotboom transformation~\cite{Liu_PNP_2014,Ding_2019}, the Nernst-Planck equations \eqref{MANP_1} can be rewritten as
\begin{equation}
\frac{\partial c^{\ell}}{\partial t}=\nabla \cdot \kappa\left[ e^{-g^\ell}\nabla (e^{g^\ell}c^\ell) \right], \label{Slotboom}
\end{equation}
where $g^\ell=q^\ell \phi + \mu^{\ell,cr}$. The scalar function $\phi$ is only mentioned in the derivation and does not need to be calculated.

Denote by $c^{\ell}_h (t) \in \mathcal{C}_{\Omega},~\ell = 0, 1, \cdots, M,$ with $c^{\ell}_h =(c^{\ell}_{i,j}) $ the spatial discretization of concentrations $c^{\ell}(x,t)$ with given $t$. Denote by $\bD_h (t)=(D^x_h,D^y_h)^T \in \mathcal{E}_{\Omega}$ with $D^x_h=(D_{i+1/2,j})$ and $D^y_h=(D_{i,j+1/2})$ the spatial discretizations of electric displacements $\bD(x,t)$ with given $t$. Central spatial difference for~\reff{Slotboom} leads to
\begin{equation} \label{spacediff}
\frac{\partial c^{\ell}_h(t)}{\partial t}
 = -\nabla_h \cdot \bJ^{\ell}(g^\ell(t),c^\ell_h(t)) \end{equation}
where the numerical flux is given by
\begin{align}
	&J^{\ell}_{i+\frac{1}{2},j}(g^\ell(t),c^\ell_h(t))=-\kappa e^{-g^{\ell}_{i+\frac{1}{2},j}(t)} D_x \left(e^{-g^\ell_h(t)}c^\ell_h(t) \right)_{i+\frac{1}{2},j}, \nonumber \\
	&J^{\ell}_{i,j+\frac{1}{2}}(g^\ell(t),c^\ell_h(t))=-\kappa e^{-g^{\ell}_{i,j+\frac{1}{2}}(t)} D_y \left(e^{-g^\ell_h(t)}c^\ell_h(t) \right)_{i,j+\frac{1}{2}}. \nonumber
\end{align}
At the half-grid points, we utilize entropic mean for interpolation
\begin{equation}
	e^{-g^{\ell}_{i+\frac{1}{2},j}(t)}=\frac{g^{\ell}_{i+1,j}(t)-g^{\ell}_{i,j}(t)}{e^{g^{\ell}_{i+1,j}(t)}-e^{g^{\ell}_{i,j}(t)}}\ \text{~and~} 
	e^{-g^{\ell}_{i,j+\frac{1}{2}}(t)}=\frac{g^{\ell}_{i,j+1}(t)-g^{\ell}_{i,j}(t)}{e^{g^{\ell}_{i,j+1}(t)}-e^{g^{\ell}_{i,j}(t)}}. \label{entropic mean}
\end{equation}
Calculation shows that the flux can be determined by
\begin{equation}\label{flux}
	J^{\ell}_{i+\frac{1}{2},j}(g^\ell(t),c^\ell(t))=-\frac{\kappa}{h}\left[ B\left( -dg^{\ell}_{i+\frac{1}{2},j}(t) \right)c^{\ell}_{i+1,j}(t)-B\left( dg^{\ell}_{i+\frac{1}{2},j}(t)\right)c^{\ell}_{i,j}(t)   \right],
\end{equation}
where $B(\cdot)$ is the Bernoulli function defined by
\begin{equation}\label{Function B}
	B(z)= \begin{cases}
	\frac{z}{e^z-1} 
               \text{~for~} z \neq 0,\\
               1, \text{~~~~~for~} z = 0,
	\end{cases}
\end{equation}
and 
\begin{align}
	dg^{\ell}_{i+\frac{1}{2},j}(t)&=g^{\ell}_{i+1,j}(t)-g^{\ell}_{i,j}(t)=q^\ell(\phi_{i+1,j}(t)-\phi_{i,j}(t))+\mu^{\ell, cr}_{i+1,j}(t)-\mu^{\ell, cr}_{i,j}(t) \nonumber \\
	& =-hq^\ell \frac{D_{i+1/2,j}(t)}{\varepsilon_{i+1/2,j}}+\mu^{\ell, cr}_{i+1,j}(t)-\mu^{\ell, cr}_{i,j}(t). \label{entropic mean 1}
\end{align}
The last equality is based on the scale relation $ \bD_h(t)=-\varepsilon \nabla_h \phi(t) $. $J^{\ell}_{i,j+1/2}(dg^\ell(t),c^\ell(t))$ can be calculated in the same way.

Alternatively, there are other interpolations, like arithmetic, geometric and harmonic mean, used to derive different numerical schemes with different functions $B(\cdot)$ \cite{Ding_2013a}. The use of harmonic mean can result in a well-conditioned system~\cite{Ding_2019}. However, the entropic mean is preferred in convection-dominated problems, resulting in the Scharfetter-Gummel scheme~\cite{Scharfetter1969IEEE}, for the reason that it can reduce to upwind scheme in large convection.

With flux~\reff{flux}, the right hand of \eqref{spacediff} can be regarded as a linear operator of $c^\ell_h(t)$. However, compared with the Laplacian, it is not symmetric. The FFT is invalid to calculate the exponential matrix function. Equivalently, we incorporate additional diffusive terms and reformulate \eqref{spacediff} as
\begin{equation}\label{s2}
	\frac{\partial c^{\ell}_h(t)}{\partial t}-\kappa \Delta_h c^\ell_h(t) + \lambda c^\ell_h(t)
	= -\nabla_h \cdot \bJ^{\ell}(dg^\ell(t),c^\ell(t)) -\kappa \Delta c^\ell_h(t) + \lambda c^\ell_h(t),
\end{equation}
where $\lambda$ is a stabilizer. 

Let $\dt$ be the time step. Denote by $c^{\ell,n}_h\in \mathcal{C}_{\Omega}$ and $\bD_h^n\in \mathcal{E}_{\Omega}$ the approximation of concentrations $c^{\ell}_h(t)$ and electric displacements $\bD_h(t)$ at time $t_n:=n\dt$. We can omit the subscript $h$ in $c^{\ell,n}_h$ and $\bD_h^n$ without causing any conflict.
Based on~\reff{s2}, the implicit first-order ETD (ETD1) scheme is designed as: for $n>0$ and given $(c^{\ell,n},\bD^n), ~\ell=1,\cdots,M$, find $c^{\ell,n+1}=\tc^{\ell}(t_{n+1})$ with periodic boundary condition such that 
\begin{align}
\frac{\partial \tc^{\ell}}{\partial t}+L_h \tc^\ell 
& = L_h c^{\ell,n}+M_h^n c^{\ell,n+1}, \label{Slotboom_2} \\ 
 \tc^{\ell}(t_n)=& c^{\ell,n}, \label{Slotboom_2_1}
\end{align}
where the linear operators $L_h$ and $M_h^n$ is defined in space $\mathcal{C}_{\Omega}$ by 
\begin{equation}
	L_h f:=\left( -\kappa\Dh+\lambda I \right) f, 
	\text{~~and~~} M_h^n f := -\nabla_h \cdot \bJ(dg^\ell(t_n),f), \quad \forall f\in \mathcal{C}_{\Omega}. \label{Matrix Mn}
\end{equation}


By introducing the integration factor, we have
\begin{equation}
\partial_t(e^{L_h t}\tc^{\ell})=e^{L_h t} L_h c^{\ell,n} + e^{L_h t} M_h^n c^{\ell,n+1}.\label{integration factor}
\end{equation}
Integrating \eqref{integration factor} from $t_n$ to $t_{n}+s$ where $0 \leq s \leq \dt$, we have the solution of \eqref{Slotboom_2}-\eqref{Slotboom_2_1} satisfying
\begin{align}
\tc^{\ell}(t_n+s) &=e^{-L_h s} c^{\ell,n} + \left(\int_0^s e^{-L_h t} dt \right)(L_h c^{\ell,n}+M^n_h c^{\ell,n+1})  \nonumber \\
& = e^{-L_h s} c^{\ell,n} +L_h^{-1}(I-e^{-L_h s})(L_h c^{\ell,n}+M^n_h c^{\ell,n+1})  \nonumber \\
& = c^{\ell,n}+L_h^{-1}(I-e^{-L_h s})M^n_h c^{\ell,n+1}. \nonumber
\end{align}
Accordingly, $c^{\ell,n+1}$ can be computed by 
\begin{equation}
	c^{\ell,n+1}
	 =c^{\ell,n}+\dt f_e(\dt L_h)M^n_h c^{\ell,n+1},\label{ETD1}
\end{equation}
where the function $f_e$ is defined as
\begin{equation}
f_e(x):=
\begin{cases}
& \frac{1-e^{-x}}{x}, \quad x \neq 0,  \\
& 1,   ~~~~~~~~\quad x=0.
\end{cases}
\label{matrix function}
\end{equation}
The matrix exponential function can be calculated efficiently by the FFT and the theory of matrix function~\reff{matrix function} is based on the following Lemma \ref{Matrix function lemma}. 
\begin{lemma}\cite{FunctionsMatrices} \label{Matrix function lemma}
Assume $\phi$ is defined on the spectrum of matrix $A \in \mathbb{C}^{m \times m}$, that is, the values
\begin{equation}
\phi^{(j)}(\lambda_i), \quad 0 \leq j \leq n_i-1, \ 1 \leq i \leq m,  \nonumber
\end{equation}
exist, where $\{ \lambda_i \}_{i=1}^m$ are the eigenvalues of $A$ and $n_i$ is the order of the largest Jordan block where $\lambda_i$ appears. Then
\begin{itemize}
\item[$(1)$] $\phi(A)$ commutes with $A$;
\item[$(2)$] $\phi(A^T)=\phi(A)^T$;
\item[$(3)$] the eigenvalues of $\phi(A)$ are $\{ \phi(\lambda_i): 1 \leq i \leq m \}$;
\item[$(4)$] $\phi(P^{-1}AP)=P^{-1}\phi(A)P$ for any nonsingular matrix $P \in \mathbb{C}^{m \times m}$;
\item[$(5)$] $\frac{d}{ds}(e^{As})=Ae^{As}=e^{As}A$ for any $s \in R$.
\end{itemize}
\end{lemma}

Note that the matrix of linear system \eqref{ETD1} has strong stiffness. The direct approaches, like Gaussian elimination, require huge memory and complexity. It is well-known that the Gaussian elimination method has the complexity of $O(n^3)$. As a result, the iterative method is more suitable here, and Picard iteration is utilized as an alternative.
It requires the matrix $L_h^{-1}(I-e^{-L_h \dt})M^n_h$ to be a contraction in the sense of a certain norm.  Some conditions of time and spatial step could be raised for contractility, which will be analyzed in Section ~\ref{section structure-preserving properties}.

\begin{remark}
    The reason why we use the implicit ETD method is to maintain physical properties in theory. Although the explicit ETD method, defined as $$c^{\ell,n+1}=c^{\ell,n}+\dt f_e(\dt L_h)M^n_h c^{\ell,n},$$ cannot be theoretically proven to maintain structure-preserving, the numerical simulation shows excellent stability and efficiency based on the FFT. In fields that require efficiency, the explicit ETD method can be used as alternative. In this way, high-order ETD and Runge-Kutta schemes can be practicable. This flexibility is one of the advantages of the ETD idea.
\end{remark}

\subsection{The numerical scheme for Maxwell-Amp\`ere equation}
Now we turn to the Maxwell-Amp\`ere equation \eqref{MANP_2}.
The electric displacement $\boldsymbol{D}^{n+1}$ should satisfy the curl-free constraint and discrete Gauss's law which is defined as 
\begin{equation}
2\kappa^2 \nabla_h \cdot \boldsymbol{D}^{n+1} = \sum_{\ell=1}^M q^\ell c^{\ell,n+1} + \rho^f. \label{discrete Gauss's law}
\end{equation}
These two properties uniquely locate the desired $\boldsymbol{D}^{n+1}$. The target here is to obtain a Gauss-law satisfying approximation of $\boldsymbol{D}^{n+1}$, denoted by $\boldsymbol{D}^*$. Then $\boldsymbol{D}^*$ will be further corrected by the local curl-free algorithm, see Section~\ref{local}. Based on the idea of ETD, we determine $\bD^*=\tilde{\bD}(t_{n+1})$ by solving
\begin{align}
&2\kappa^2 \tilde{\bD}_t=\sum_{\ell=1}^M q^\ell\left(\kappa\nabla_h \tc^{\ell}-\bJ^{\ell,n}- \kappa\nabla_h c^{\ell,n} \right)+\tilde{\boldsymbol{F}}+\boldsymbol{\Theta}^n, \label{D1}  \\
& \tilde{\bD}(t_n)=\bD^n, \nonumber
\end{align}
where $\tilde{\boldsymbol{F}}(t)$ is an undetermined correction mesh function designed to ensure the Gauss's law and $\boldsymbol{\Theta}^n$ is a divergence-free vector. Calculation the discrete divergence of \eqref{D1} and integration from $t_n$ to $t_{n+1}$ lead to
\begin{equation}\label{ma2}
2\kappa^2 \nabla_h \cdot (\bD^*-\bD^n)=\sum_{\ell=1}^M q^\ell \left( c^{\ell,n+1}-c^{\ell,n}\right) + \sum_{\ell=1}^M q^\ell \kappa \left(  \int_{t_n}^{t_{n+1}} \tc^{\ell}dt- \dt c^{\ell,n} \right) +  \nabla_h \cdot \int_{t_n}^{t_{n+1}} \tilde{\boldsymbol{F}} dt. 
\end{equation}
To maintain the Gauss's law, the second and third terms of the right hand of~\reff{ma2} should be vanished. Therefore, we denote $\boldsymbol{F}^n:=\int_{t_n}^{t_{n+1}} \tilde{\boldsymbol{F}} dt$ with periodic boundary condition determined by
\begin{align}
& - \Delta_h \psi^n = \sum_{\ell=1}^M q^\ell \kappa \left(  \int_{t_n}^{t^{n+1}} \tc^{\ell}dt- \dt c^{\ell,n} \right),\nonumber \\
& \boldsymbol{F}^n=\nabla_h \psi^n. \nonumber 
\end{align}
This Laplacian can be solved efficiently by the FFT. In fact, any alternative $\boldsymbol{F}^n$ satisfying 
\begin{equation}
\sum_{\ell=1}^M q^\ell \kappa \left(  \int_{t_n}^{t^{n+1}} \tc^{\ell}dt- \dt c^{\ell,n} \right) +  \nabla_h \cdot \boldsymbol{F}^n = 0 \nonumber  
\end{equation}
is acceptable. With the correction $\boldsymbol{F}^n$, we can obtain $\boldsymbol{D}^*$:
\begin{equation}
2\kappa^2 (\boldsymbol{D}^*-\boldsymbol{D}^n)=\sum_{\ell=1}^M q^\ell\left(\kappa\nabla_h \int_{t_n}^{t^{n+1}} \tc^{\ell}dt - \dt \bJ^{\ell,n} -\kappa\dt \nabla_h c^{\ell,n}\right)+ \boldsymbol{F}^n+\dt\boldsymbol{\Theta}^n. \label{D* equation}
\end{equation}
$\boldsymbol{\Theta}^n$ can be initialized as any divergence-free vector since we have a further process to correct $\boldsymbol{D}^*$. However, an appropriate $\boldsymbol{\Theta}^n$ leads to more precise $\boldsymbol{D}^*$ that is closer to the desired $\boldsymbol{D}^{n+1}$. Several choices have been tested in \cite{Qiao_MANP_model_2023} and we use $\boldsymbol{\Theta}^n$ based on the same idea:
\begin{equation}
\dt \boldsymbol{\Theta}^n = 2\kappa^2(\boldsymbol{D}^n-\boldsymbol{D}^{n-1})-\sum_{\ell=1}^M q^\ell\left(\kappa\nabla_h \int_{t_{n-1}}^{t^n} \tc^{\ell} dt -\dt \bJ^{\ell,n-1} - \kappa\dt \nabla_h c^{\ell,n-1}\right)-\boldsymbol{F}^{n-1}. \nonumber
\end{equation}
The combination of \eqref{ETD1} and \eqref{D* equation} is the first part of our numerical method. Then we will introduce the second part to get $\bD^{n+1}$. 

\subsection{Local curl-free algorithm}~\label{local}
By numerical scheme \eqref{ETD1} and \eqref{D* equation}, we have got $c^{\ell,n+1}$ and $\bD^*$ close to $\bD^{n+1}$. The correction function $\boldsymbol{F}$ ensures that the approximated electric displacement $\bD^*$ maintains the Gauss's law. Furthermore, $\bD^*$ also needs to satisfy the curl-free constraint. To address this, a local curl-free algorithm has been proposed with Gauss-law satisfying property starting from $\boldsymbol{D}^*$. Initially utilized for simulating charged particles, this algorithm modifies the dynamics to enable Coulombic equilibrium using a local update \cite{MR:PRL:2002}. It has also been applied to numerically solve the MANP equations, as referenced in conjunction with a decoupled semi-implicit finite difference method \cite{Qiao_NumericalMANP_2023}. Let us now recall the main steps of this algorithm.

The idea is that the discrete curl-free condition is equivalent to a convex constraint optimization problem
\begin{equation}
\bD^{n+1} = \min_{\bD \in \mathcal{E}_{\Omega}} \mathcal{F}_{pot}(\bD):=h^2\kappa^2 \sum_{i,j} \left( \frac{D^2_{i+\frac{1}{2},j}}{\varepsilon_{i+\frac{1}{2},j}} + \frac{D^2_{i,j+\frac{1}{2}}}{\varepsilon_{i,j+\frac{1}{2}}} \right),\quad \text{s.t.}\quad  2\kappa^2 \nabla_h \cdot \bD =\rho^{n+1},  \nonumber
\end{equation}
where $\rho^{n+1}=\sum_{\ell=1}^M q^\ell c^{\ell,n+1} + \rho^f$.
Introducing the Lagrangian
\begin{equation}
\mathcal{L}(\bD, \phi) := h^2\kappa^2 \sum_{i,j} \left( \frac{D^2_{i+\frac{1}{2},j}}{\varepsilon_{i+\frac{1}{2},j}} + \frac{D^2_{i,j+\frac{1}{2}}}{\varepsilon_{i,j+\frac{1}{2}}} \right) + h^2 \sum_{i,j} \phi_{i,j}\left( 2\kappa^2\nabla_h \cdot \bD - \rho^{n+1} \right)_{i,j}, \nonumber
\end{equation}
where $\phi$ is the Lagrange multiplier. it is straightforward that the minimizer $\bD^{n+1}\in\mathcal{E}_{\Omega}$ satisfies
\begin{equation}
\bD^{n+1} = - \varepsilon \nabla_h \phi, \quad 2\kappa^2\nabla_h \cdot \bD^{n+1} = \rho^{n+1}. \nonumber
\end{equation}
Accordingly, the curl-free condition $\nabla_h \times (\bD / \varepsilon)=0$ is fulfilled. By the convexity of functional $\mathcal{F}_{pot}$, the minimizer is unique. The above derivation provides an optimization approach starting from $\bD^*$, which is given in \eqref{D* equation}, to reach the curl-free $\bD^{n+1}$. 
In detail, the local relaxation updates the electric displacement in every mesh cell successively to minimize $\mathcal{F}_{pot}$ while maintaining the Gauss's law. For example, consider a single cell $(i+\delta_i,j+\delta_j)$, $\delta_i,\delta_j\in \{0,1\}$. The displacements defined on four edges are $D_{i+\hf,j}$, $D_{i,j+\hf}$, $D_{i+\hf,j+1}$ and $D_{i+1,j+\hf}$. We can update the electric displacements by
\begin{align}
D_{i+\hf,j} & \leftarrow D_{i+\hf,j}+\frac{\eta}{h}, \nonumber \\
D_{i,j+\hf} & \leftarrow D_{i,j+\hf}-\frac{\eta}{h}, \nonumber \\
D_{i+\hf,j+1} & \leftarrow D_{i+\hf,j+1}-\frac{\eta}{h}, \nonumber \\
D_{i+1,j+\hf} & \leftarrow D_{i+1,j+\hf}+\frac{\eta}{h}, \nonumber 
\end{align}
where $\eta$ is determined by the local minimizer of the energy $\mathcal{F}_{pot}$ and has an explicit expression
\begin{equation}
\frac{\eta}{h} = -\frac{\frac{D_{i+\hf,j}}{\varepsilon_{i+\hf,j}} -  \frac{D_{i+\hf,j+1}}{\varepsilon_{i+\hf,j+1}} + \frac{D_{i+1,j+\hf}}{\varepsilon_{i+1,j+\hf}} -\frac{D_{i,j+\hf}}{\varepsilon_{i,j+\hf}} }{\frac{1}{\varepsilon_{i+\hf,j}}+\frac{1}{\varepsilon_{i+\hf,j+1}}+\frac{1}{\varepsilon_{i,j+\hf}}+\frac{1}{\varepsilon_{i+1,j+\hf}}}. \nonumber
\end{equation}
The discrete Gauss’s law is strictly maintained because of the same flux of inflow and outflow during the update. Such iteration traverses every cell until the energy $\mathcal{F}_{pot}$ decreases below the tolerance.
\begin{remark}
    The curl-free relaxation essentially minimizes the energy functional through local method while maintaining the constraint. It is a new approach to solving the Poisson's equation coupled with the dynamics of mass. For example, in [27], the Poisson's part in the Vlasov-Poisson equations is equivalently reformulated as Amp\`ere system and is numerically solved based on a group of curl-free basis. The Vlasov-Amp\`ere system can also be treated efficiently by the local curl-free relaxation.
\end{remark}

We summarize the whole numerical method for the MANP equations in Algorithm \ref{algorithm:MANP}.
\begin{algorithm}\label{algorithm:MANP}
\begin{algorithmic}[1]
    \Require Final time $T$, tolerance $\varepsilon_{\text{tol}}$ for curl-free relaxation, initial value $ c^{\ell,0}>0, ~\ell = 1,\cdots,M$, and $\bD^0$ satisfying the discrete Gauss's law.
	\State Solve the implicit ETD scheme \eqref{ETD1} by Picard iteration and obtain $ c^{\ell,n+1}, ~\ell = 1,\cdots,M$ and calculate $\boldsymbol{\Theta}^n$ and $\boldsymbol{F}^n$.
	\State Based on $c^{\ell,n+1}$, $\boldsymbol{\Theta}^n$ and $\boldsymbol{F}^n$, calculate the the interim electric displacement $\bD^*$ in \eqref{D* equation}.
	\State Perform curl-free relaxation starting from $\bD^*$ to obtain $\bD^{n+1}$
	\State If $T \leq t^{n+1}$ then return step 1, otherwise stop.  
\end{algorithmic}
\caption{The implicit ETD1 method for the MANP equations}
\end{algorithm}

\section{Analysis on structure-preserving properties}  \label{section structure-preserving properties}
\subsection{Unique solvability}
Now we prove the unique solvability of implicit ETD scheme \eqref{ETD1} which can be regarded as a fixed point problem. The Picard iteration is utilized for the reason that the matrix needs to be solved has strong stiffness. Note that it requires a contractive property and that is not unconditional. The following lemma and theorem are proposed to analyze the constraint.  

For any linear operators $A$, the infinite norm is defined as
$$
\| A \|_\infty := \max_{\|v\|_\infty=1} \| Av \|_\infty = \max_{i} \sum_{j} |A_{i,j}|.
$$
The first lemma is the estimation for the matrix function $E:=\dt f_e(L_h)$ where the function $f_e$ is defined in \eqref{matrix function}.

\begin{lemma}\label{infinite norm estimate lemma}
Given $L_h$ defined in \eqref{Matrix Mn} and function $f_e$ defined in \eqref{matrix function}, the matrix $f_e(\dt L_h)$ is positive and has the following estimate
\begin{equation}
\| \dt f_e(\dt L_h) \|_\infty = \frac{1}{\lambda}\left(1-e^{-\lambda \dt}\right) \leq \dt. \label{infinite norm estimate}
\end{equation}
\end{lemma}
\begin{proof}
 For the positivity, since $ \dt f_e(\dt L_h)=\int_0^\dt e^{-L_h t} dt$, we only need to prove the positivity of the matrix exponential function $e^{-L_h}$. It follows from the definition of $e^{-L_h}$ that
\begin{align}
e^{-L_h} &= e^{\kappa \Delta_h - \lambda I}=e^{-\lambda} e^{\kappa \Delta_h} \nonumber \\
& = e^{-\lambda} e^{\kappa(\alpha I + Q)} \nonumber \\
& = e^{-\lambda+\kappa\alpha} e^{\kappa Q}, \nonumber
\end{align}
where $\alpha I$ is the diagonal part of the Laplacian with $\alpha = -2$ in 1-D and $\alpha = -4$ in 2-D. $Q$ denotes the non-diagonal part which is a non-negative matrix with component either zero or one. The series expansion
\begin{equation}
e^{\kappa Q} = I+\sum_{k=1}^\infty \frac{1}{k!}(\kappa Q)^k, \nonumber
\end{equation}
ensures the positivity of $e^{\kappa Q}$. With an additional coefficient, the positivity of $e^{-L_h}$ is proved. 

Then we prove the estimate \eqref{infinite norm estimate} of $f_e(\dt L_h)$. Let $\boldsymbol{1}$ be the vector whose every component is one. Since the positivity of $f_e(\dt L_h)$ has already proved, the infinite norm can be calculated directly as 
\begin{align}
\|\dt f_e(\dt L_h)\|_\infty & = \max_i (\dt f_e(\dt L_h) \cdot \boldsymbol{1})_i = \max_i \int_0^\dt (e^{-L_h t}\cdot \boldsymbol{1})_i dt \nonumber \\
& = \max_i \int_0^\dt (e^{-\lambda t} e^{\kappa \Delta_h t}\cdot \boldsymbol{1})_i dt \nonumber \\
& = \int_0^\dt e^{-\lambda t} dt. \nonumber
\end{align}
The last equation is based on the expansion of the exponential function:
\begin{equation}
e^{\kappa \Delta_h t}\cdot \boldsymbol{1} = \boldsymbol{1} + \sum_{k=1}^\infty \frac{1}{k!} (\kappa t \Delta_h)^k \cdot \boldsymbol{1} = \boldsymbol{1}. \nonumber
\end{equation}
Then we have \eqref{infinite norm estimate} and complete the proof.
\end{proof}

With the infinite norm estimation in Lemma \ref{infinite norm estimate lemma}, we are ready to give the following theorem for unique solvability.
\begin{theorem}\label{solvability theorem}
Given $c^{\ell,n}$ and $\bD^n$, the implicit ETD1 scheme \eqref{ETD1} is a linear fixed point problem. If time step $\dt$ and space step $h$ satisfy the condition
\begin{equation}
\frac{\dt}{h^2} < \frac{1}{16\kappa B\left(-\max_{i,j,\ell}\left(|dg^{\ell,n}_{i+1/2,j}|,|dg^{\ell,n}_{i,j+1/2}| \right)\right)},   \label{CFL condition 1}
\end{equation}
then the matrix $\dt f_e(\dt L_h)M^n_h$ is contractive map, which means the fixed point is unique and the Picard iteration is valid as well.
\end{theorem}
\begin{proof}
This theorem is objected to estimate the norm of $\dt f_e(\dt L_h) M_h^n$ where $M_h^n$ is given in \eqref{Matrix Mn}. With Lemma \ref{infinite norm estimate lemma} for infinite norm of $\dt f_e(\dt L_h)$, we only need to estimate $\| M_h^n \|_\infty$. For simplicity, we use the index notation
$$
[i,j]:=(i-1)N_y+j, \quad i=1,\cdots ,N_x,\  j=1,\cdots,N_y,
$$
for periodic boundary condition:
\begin{equation}
\begin{array}{l}
i^{+}=\left\{\begin{array}{l}
i+1 \text {~~~for~~} i=1, \ldots, N_x-1, \\
1 \text {~~~~for~~} i=N_x,
\end{array} \qquad i^{-}=\left\{\begin{array}{l}
i-1 \text {~~~for~~} i=2, \ldots, N_x, \\
N_x \text {~~~~for~~} i=1,
\end{array}\right.\right. \\
j^{+}=\left\{\begin{array}{l}
j+1 \text {~~~for~~} j=1, \ldots, N_y-1, \\
1 \text {~~~~for~~} j=N_y,
\end{array} \qquad j^{-}=\left\{\begin{array}{l}
j-1 \text{~~~for~~} j=2, \ldots, N_y, \\
N_y \text{~~~~for~~} j=1.
\end{array}\right.\right. \\
\end{array} \nonumber
\end{equation}
We study the elements of each row of $M_h^n$. The non-zero entries of the $k$-th $(k=[i,j])$ row are given by
\begin{equation}
\frac{h^2}{\kappa}(M_h^n)_{k, m} = \left\{
\begin{array}{ll}
B(dg^{\ell,n}_{i-1/2,j}), &  m =\left[i^-, j\right]\\
B(dg^{\ell,n}_{i,j-1/2}), &  m =\left[i, j^-\right]\\
-B(dg^{\ell,n}_{i+1/2,j})-B(-dg^{\ell,n}_{i-1/2,j})-B(dg^{\ell,n}_{i,j+1/2})-B(-dg^{\ell,n}_{i,j-1/2}), & m=k  \\
B(-dg^{\ell,n}_{i,j+1/2}), & m =\left[i^+, j\right]\\
B(-dg^{\ell,n}_{i+1/2,j}), & m =\left[i, j^+\right]
\end{array}
\right. \nonumber
\end{equation}
where $B(z)$ and $dg^{\ell,n}$ are defined in \eqref{Function B} and \eqref{entropic mean 1} respectively. The estimation is straightforward
\begin{align}\label{estimate M}
\| M_h^n \|_\infty = \max_i \sum_j |(M_h^n)_{i,j}| \leq \frac{8\kappa}{h^2}B\left(-\max_{i,j,\ell}\left(\left|dg^{\ell,n}_{i+\frac{1}{2},j}\right|,\left|dg^{\ell,n}_{i,j+\frac{1}{2}}\right| \right)\right),
\end{align}
where the monotonicity of function $B(z)$ is used in the last inequality. Now by direct calculation, we have
\begin{align}
\|\dt f_e(\dt L_h) M_h^n\|_\infty & \leq \|\dt f_e(\dt L_h)\|_\infty \times \|M_h^n\|_\infty  \nonumber \\
& \leq \frac{8\kappa \dt}{h^2}B\left(-\max_{i,j,\ell}\left(\left|dg^{\ell,n}_{i+\frac{1}{2},j}\right|,\left|dg^{\ell,n}_{i,j+\frac{1}{2}}\right| \right)\right) < \frac{1}{2} <1.\nonumber
\end{align}
Therefore, $\dt f_e(\dt L_h) M_h^n$ is a contraction in the sense of the infinite norm and has unique fixed point which is equivalent to the solution of \eqref{ETD1}.
\end{proof} 

\subsection{Mass conservation}
By calculating the numerical integral in both sides of \eqref{Slotboom_2} and applying the summation by parts formula, we have
\begin{align}
& \frac{\partial}{\partial t} (\tc^{\ell},1)+\lambda (\tc^{\ell},1)=\lambda (c^{\ell,n},1), \nonumber \\
& (\tc^{\ell}(t_n),1)=(c^{\ell,n},1).  \nonumber
\end{align}
Equivalently, it is an ODE system of $(\tc^{\ell},1)$ which can be denoted as $\bar{\tc}^{\ell}(t)$:
\begin{align}
& \frac{\partial}{\partial t} \bar{\tc}^{\ell} + \lambda \bar{\tc}^{\ell}=\lambda \bar{c}^{\ell,n}, \label{mass ODE1} \\
& \bar{\tc}^{\ell}(t_n)=\bar{c}^{\ell,n}.  \label{mass ODE2}
\end{align}
The ODE theory shows that equations \eqref{mass ODE1} and \eqref{mass ODE2} have a unique solution. Furthermore, it is observed that the constant function $\bar{\tc}^{\ell}(t) \equiv \bar{c}^{\ell,n}$ satisfies \eqref{mass ODE1} and \eqref{mass ODE2}. We can conclude that the unique solution is the constant function and 
\begin{equation}
\bar{c}^{\ell,n+1} = \bar{\tc}^{\ell}(t_{n+1}) = \bar{c}^{\ell,n}.\nonumber
\end{equation}
Then the mass conservation is satisfied.

\subsection{Positivity-preserving property}
Assuming that the conditions of Theorem \ref{solvability theorem} hold, we can rewrite \eqref{ETD1} as 
\begin{equation}
c^{\ell,n+1} = \left(I-\dt f_e(\dt L_h) M_h^n\right)^{-1} c^{\ell,n}.  \nonumber
\end{equation}
It has been proved that the matrix $I-\dt M_h^n$ is an M-matrix with $\left(I-\dt M_h^n\right)^{-1} \succ 0$~\cite{Qiao_NumericalMANP_2023}. As a result, if the difference between $\dt f_e(\dt L_h)$ and $\dt I$ is small enough, the matrix $I-\dt f_e(\dt L_h) M_h^n $ should be M-matrix as well. The following Lemma \ref{inverse operator continuity} gives the key that describes the continuity of the inverse operator.

\begin{lemma}\cite{MatrixInverse}  \label{inverse operator continuity}
For any matrix norm $\|\cdot\|$ with $\|I\|=1$, if $\|A^{-1}\| \|E\|<1$, then $A+E$ is invertible and 
\begin{equation}
\|(A+E)^{-1}-A^{-1}\| \leq \frac{\|A^{-1}\|^2\|E\|}{1-\|A^{-1}\| \|E\|}. \nonumber
\end{equation}
\end{lemma}
Define $$
\alpha_{\min}:=\min_{i,j} \left[\left(I-\dt M_h^n\right)^{-1}\right]_{i,j}>0.
$$
Based on the definition of the infinite norm, it is noticed that if 
\begin{equation}
	\| \left(I-\dt f_e(\dt L_h) M_h^n\right)^{-1}-\left(I-\dt M_h^n\right)^{-1} \|_\infty \leq \alpha_{\min} \label{MatrixInverse 1}
\end{equation}
holds, we can get the positivity of $\left(I-\dt f_e(\dt L_h) M_h^n\right)^{-1}$. Furthermore, substituting $A$ with $I-\dt M_h^n$ and $E$ with $\dt(f_e(\dt L_h)-I) M_h^n$, Lemma~\ref{inverse operator continuity} tells that 
\begin{equation}
    \| \left(I-\dt f_e(\dt L_h) M_h^n\right)^{-1}-\left(I-\dt M_h^n\right)^{-1} \|_\infty \leq \frac{\|\left(I-\dt M_h^n\right)^{-1}\|_\infty^2 \|\dt(f_e(\dt L_h)-I) M_h^n\|_\infty}{1-\|\left(I-\dt M_h^n\right)^{-1}\|_\infty \|\dt(f_e(\dt L_h)-I) M_h^n\|_\infty}. \nonumber
\end{equation}
So that a sufficient condition of \eqref{MatrixInverse 1} is
\begin{equation}
	\frac{\|\left(I-\dt M_h^n\right)^{-1}\|_\infty^2 \|\dt(f_e(\dt L_h)-I) M_h^n\|_\infty}{1-\|\left(I-\dt M_h^n\right)^{-1}\|_\infty \|\dt(f_e(\dt L_h)-I) M_h^n\|_\infty} <\alpha_{\min}, \nonumber
\end{equation}
or equivalently
\begin{equation}
	\|\dt(f_e(\dt L_h)-I) M_h^n\|_\infty < \frac{\alpha_{\min}}{\|\left(I-\dt M_h^n\right)^{-1}\|_\infty^2+\alpha_{\min}\|\left(I-\dt M_h^n\right)^{-1}\|_\infty }.\label{positivity target 2}
\end{equation}
We remark that if \reff{positivity target 2} holds, the condition $\|\left(I-\dt M_h^n\right)^{-1}\|_\infty\|\dt(f_e(\dt L_h)-I) M_h^n\|_\infty<1$ of Lemma~\ref{inverse operator continuity} is satisfied naturally. The following Theorem \ref{PositivityTheorem} gives the proof of \eqref{positivity target 2} and corresponding conditions.
\begin{theorem} \label{PositivityTheorem}
Given $c^{\ell,n}>0$ and $\bD^n$, if time step $\dt$ and space step $h$ satisfy the condition \eqref{CFL condition 1} and
\begin{equation}
\left(\frac{4 \kappa }{h^2}+\frac{\lambda }{4} \right) \dt < \frac{\alpha_{\min}}{4+2\alpha_{\min}}, \label{final constraint}
\end{equation}
then the positivity $c^{\ell,n+1}>0, \ell=1,\cdots,M $ is preserved.
\end{theorem}

\begin{proof}
	To achieve \eqref{positivity target 2}, we need to estimate $\|\dt f_e(\dt L_h)-\dt I\|_\infty$ and $\|M_h^n\|_\infty$. On the one hand, from \eqref{CFL condition 1}-\eqref{estimate M}, it has been shown that 
	\begin{equation}
		\dt \|M_h^n\|_\infty < 1/2.  \label{Mn_infty}
	\end{equation}
On the other hand, matrix $\dt (f_e(\dt L_h)-I)$ is symmetric and has negative diagonal elements and positive off-diagonal elements, provided by the positivity of $f_e(\dt L_h)$ and $\|\dt f_e(\dt L_h)\|_\infty < \dt$ in the proof of Lemma \ref{infinite norm estimate lemma}. Then
\begin{align}
	\| \dt f_e(\dt L_h)-\dt I \|_\infty = \frac{1-e^{-\lambda \dt}}{\lambda}-\dt (2\min_i f_e(\dt L_h)_{i,i}-1). \label{E-s1} 
\end{align}
Furthermore, the symmetry of $f_e(\dt L_h)$ means that the diagonal elements are located between the maximum and minimum eigenvalues, i.e.
\begin{equation}
	\min_i f_e(\dt L_h)_{i,i} \geq \min \lambda( f_e(\dt L_h)) , \label{E-s2}
\end{equation}
where $\lambda(f_e(\dt L_h))$ presents the eigenvalues of $f_e(\dt L_h)$.
With Lemma \ref{Matrix function lemma} and periodic boundary condition, the eigenvalues of $\dt f_e(\dt L_h)$ can be calculated explicitly
\begin{equation}
	\lambda(\dt f_e(\dt L_h))=\left\{ \dt f_e \left( \frac{4\kappa \dt}{h^2}(\sin^2\frac{k\pi}{N}+\sin^2\frac{l\pi}{N})+\lambda \dt \right), \ k,l=1,\cdots,N \right\}. \label{E-s3}
\end{equation}
Noticing that $f_e$ is Lipschitz continuous in $[0,+\infty)$
\begin{equation}
	|f_e(x)-f_e(y)| \leq \frac{1}{2}|x-y|, \quad x,y \in [0,+\infty), \nonumber
\end{equation}
we have
\begin{align}
	\| \dt f_e(\dt L_h)-\dt I \|_\infty& = \dt \left( f_e(\lambda \dt)+f_e(0)-2f_e\left(\frac{8 \kappa \dt}{h^2}+\lambda \dt\right) \right) \nonumber \\
	& \leq \dt \left( \frac{8 \kappa \dt}{2h^2}+\frac{8 \kappa \dt}{2h^2}+\frac{\lambda \dt}{2} \right)  = \dt \left( \frac{8 \kappa \dt}{h^2}+\frac{\lambda \dt}{2} \right).   \label{E-s4} 
\end{align}
With combination of \eqref{E-s1}-\eqref{E-s4}, we obtain the estimation for $\|\dt f_e(\dt L_h)-\dt I\|_\infty$:
\begin{align}
\| \dt f_e(\dt L_h)-\dt I \|_\infty &= \frac{1-e^{-\lambda \dt}}{\lambda}-\dt (2\min_i f_e(\dt L_h)_{i,i}-1)  \nonumber \\
& \leq \frac{1-e^{-\lambda \dt}}{\lambda} +\dt -2\dt \min \lambda( f_e(\dt L_h)) \nonumber \\
& = \frac{1-e^{-\lambda \dt}}{\lambda} +\dt -2 \dt f_e\left(\frac{8 \kappa \dt}{h^2}+\lambda \dt\right) \nonumber \\
& = \dt \left( f_e(\lambda \dt)+f_e(0)-2f_e\left(\frac{8 \kappa \dt}{h^2}+\lambda \dt\right) \right) \nonumber \\
& \leq \dt \left( \frac{8 \kappa \dt}{2h^2}+\frac{8 \kappa \dt}{2h^2}+\frac{\lambda \dt}{2} \right)  \nonumber \\
& = \dt \left( \frac{8 \kappa \dt}{h^2}+\frac{\lambda \dt}{2} \right).   \label{E-sI} 
\end{align}
Finally, it follows from \eqref{Mn_infty} and \eqref{E-sI} that
\begin{equation}
\| \dt (f_e(\dt L_h)- I) M_h^n\|_\infty \leq \|M_h^n\|_\infty \| \dt f_e(\dt L_h)-\dt I \|_\infty \leq \frac{4 \kappa \dt}{h^2}+\frac{\lambda \dt}{4}. \label{half conclusion}
\end{equation}
A derivation of the right side of target \eqref{positivity target 2} leads to
\begin{align}
\frac{\alpha_{\min}}{\|\left(I-\dt M_h^n\right)^{-1}\|_\infty^2+\alpha_{\min}\|\left(I-\dt M_h^n\right)^{-1}\|_\infty } & \geq \frac{\alpha_{\min}\left( 1-\dt \|M_h^n\|_\infty \right)^2}{1+\alpha_{\min}\left( 1-\dt \|M_h^n\|_\infty \right)} \nonumber \\
& \geq \frac{\alpha_{\min}}{4+2\alpha_{\min}}, \label{approximation C_M}
\end{align}
where the inequality 
$$
\|\left(I-\dt M_h^n\right)^{-1}\|_\infty \leq \frac{1}{1-\dt \|M_h^n\|_\infty}
$$
is applied. A combination of \eqref{half conclusion} and \eqref{approximation C_M} gives the condition \eqref{final constraint}.
\end{proof}

\subsection{Energy dissipation}
Energy dissipation is an essential property in both physical modeling and numerical simulation. Following the discrete inner product in Definition 2.4, the electrostatic free energy $\mathcal{F}$ in \eqref{original energy} can be discretized by
\begin{equation}
\mathcal{F}_h^n = \sum_{\ell=1}^n \left( c^{\ell,n}\log(c^{\ell,n}) + \mu^{\ell,\mathrm{cr},n} , 1 \right) + \left( \frac{\kappa^2|\bD^n|^2}{\varepsilon}, 1 \right). \nonumber
\end{equation}
The theorem follows a similar idea proposed in \cite{Efficient_2021} and \cite{Qiao_NumericalMANP_2023}, with modification to the matrix functions and electric displacement.
\begin{theorem}\label{Energy dissipation theorem}
Let $\varepsilon_{\min}=\min_{i,j} \{ \varepsilon_{i+1/2,j},\varepsilon_{i,j+1/2} \}$, $\varepsilon_{\max}=\max_{i,j} \{ \varepsilon_{i+1/2,j},\varepsilon_{i,j+1/2} \}$ and $c_{\max}=\max_{i,j,\ell} \{ c^{\ell,n+1}_{i,j} \}$. Assume the conditions \eqref{CFL condition 1} and \eqref{final constraint} hold. If $\mu^{\ell,\mathrm{cr}}$ is independent in time and 
\begin{equation}
\dt+\frac{8\kappa }{\lambda h^2}\dt \leq \frac{2\varepsilon_{\min}^3\kappa}{c_{\max}\varepsilon_{\max}^2\sum_{\ell=1}^M |q^\ell|^2} \exp\left[ -\max_{i,j,\ell}\left( \left|dg^{\ell,n}_{i+\frac{1}{2},j}\right|,\left|dg^{\ell,n}_{i,j+\frac{1}{2}}\right| \right) \right], \label{Energy dissipation condition}
\end{equation} 
then the solution of numerical scheme \eqref{ETD1} satisfies the discrete energy dissipation
\begin{equation}
\mathcal{F}_h^{n+1}-\mathcal{F}_h^n \leq -\frac{\dt}{2} I_1, \label{discrete energy decay}
\end{equation}
where
\begin{align}
I_1=\sum_{\ell=1}^M \kappa \ciptwo{f_e(L_h\dt)\left[ e^{-g^{\ell,n}} \nabla_h (e^{g^{\ell,n}}c^{\ell,n+1}) \right]}{\nabla_h \log(e^{g^{\ell,n}}c^{\ell,n+1})} \geq 0 \nonumber.
\end{align}
\end{theorem}

\begin{proof}
Based on the difference of energy between two adjacent time steps and the time-independent $\mu^{\ell,\mathrm{cr}}$, we have
\begin{align}
\mathcal{F}_h^{n+1}-\mathcal{F}_h^n = &\sum_{\ell=1}^M \left[ \ciptwo{c^{\ell,n+1}-c^{\ell,n}}{\log c^{\ell,n+1}+g^{\ell,n}-q^{\ell}\phi^n} + \ciptwo{c^{\ell,n}}{\log \frac{c^{\ell,n+1}}{c^{\ell,n}}} \right] \nonumber \\
&+ \ciptwo{|\bD^{n+1}|^2-|\bD^n|^2}{\frac{\kappa^2}{\varepsilon}}. \nonumber
\end{align}
The mass conservation leads to
\begin{align}
\ciptwo{c^{\ell,n}}{\log \frac{c^{\ell,n+1}}{c^{\ell,n}}} \leq \ciptwo{c^{\ell,n+1}-c^{\ell,n}}{1}=0. \nonumber
\end{align}
Then we have
\begin{align}
\mathcal{F}_h^{n+1}-\mathcal{F}_h^n \leq -\dt I_1+\dt^2 I_2, \nonumber
\end{align}
where 
\begin{equation}
I_1=-\sum_{\ell=1}^M \ciptwo{\frac{c^{\ell,n+1}-c^{\ell,n}}{\dt}}{\log c^{\ell,n+1}+g^{\ell,n}}, \nonumber
\end{equation}
\begin{equation}
I_2=\frac{1}{\dt^2}\ciptwo{|\bD^{n+1}|^2-|\bD^n|^2}{\frac{\kappa^2}{\varepsilon}} - \frac{1}{\dt^2}  \ciptwo{\sum_{\ell=1}^M q^{\ell}(c^{\ell,n+1}-c^{\ell,n}) }{\phi^n}.  \nonumber
\end{equation}
By the numerical scheme \eqref{ETD1} and summation by parts formula, we have
\[
I_1=-\sum_{\ell=1}^M \ciptwo{f_e(L_h\dt) \nabla_h \cdot \kappa \left[ e^{-g^{\ell,n}} \nabla_h (e^{g^{\ell,n}}c^{\ell,n+1}) \right]}{\log(e^{g^{\ell,n}}c^{\ell,n+1})} .
\]
 It's clear that $f_e(L_h\dt)$ can be expanded into an infinite series of $\Delta_h$, based on the definition of $L_h$. The interchangeability of discrete Laplacian and divergence operator is valid in central difference frame: for any $\boldsymbol{u}=(u^x, u^y)^T \in \mathcal{E}_{\Omega}$ with periodic boundary condition
\begin{equation}
\Delta_h (\nabla_h \cdot \boldsymbol{u}) = \nabla_h \cdot (\Delta_h \boldsymbol{u}), \quad \text{where  } \Delta_h \boldsymbol{u}=(\Delta_h u^x, \Delta_h u^y)^T. \nonumber
\end{equation}
Therefore, 
\begin{align}
	I_1
	=\sum_{\ell=1}^M \kappa \ciptwo{f_e(L_h\dt)\left[ e^{-g^{\ell,n}} \nabla_h (e^{g^{\ell,n}}c^{\ell,n+1}) \right]}{\nabla_h \log(e^{g^{\ell,n}}c^{\ell,n+1})}. \label{e proof 1}
\end{align}
 Furthermore, we have already proved that $f_e(L_h\dt)$ is a positive matrix, which leads to
\begin{equation}
I_1 \geq  \sum_{\ell=1}^M \kappa \min_i f_e(L_h\dt)_{i,i} \ciptwo{ e^{-g^{\ell,n}} \nabla_h (e^{g^{\ell,n}}c^{\ell,n+1}) }{\nabla_h \log(e^{g^{\ell,n}}c^{\ell,n+1})} \geq 0. \label{e proof 2}
\end{equation}
The last inequality of \eqref{e proof 2} comes from the monotonicity of exponential function that
\begin{align}
\ciptwo{ e^{-g^{\ell,n}} \nabla_h (e^{g^{\ell,n}}c^{\ell,n+1}) }{\nabla_h \log(e^{g^{\ell,n}}c^{\ell,n+1})}
  &=\sum_{i,j} e^{-g^{\ell,n}_{i+1/2,j}}(e^{\mu^{\ell,n,*}_{i+1,j}}- e^{\mu^{\ell,n,*}_{i,j}})(\mu^{\ell,n,*}_{i+1,j}-\mu^{\ell,n,*}_{i,j}) \nonumber \\
+\sum_{i,j} e^{-g^{\ell,n}_{i,j+1/2}}&(e^{\mu^{\ell,n,*}_{i,j+1}}- e^{\mu^{\ell,n,*}_{i,j}})(\mu^{\ell,n,*}_{i,j+1}-\mu^{\ell,n,*}_{i,j}), \geq 0, \nonumber
\end{align}
where $\mu^{\ell,n,*}=\log(e^{g^{\ell,n}}c^{\ell,n+1})$.
It follows from the discrete Gauss's law \eqref{discrete Gauss's law} that
\begin{align}
I_2 &= \frac{1}{\dt^2}\ciptwo{|\bD^{n+1}|^2-|\bD^n|^2}{\frac{\kappa^2}{\varepsilon}} - \frac{2\kappa^2}{\dt^2} \ciptwo{\nabla_h \cdot (\bD^{n+1}-\bD^{n})}{\phi^n} \nonumber \\
&=\frac{1}{\dt^2}\ciptwo{|\bD^{n+1}|^2-|\bD^n|^2}{\frac{\kappa^2}{\varepsilon}} - \frac{2\kappa^2}{\dt^2} \ciptwo{\bD^{n+1}-\bD^{n}}{\frac{1}{\varepsilon} \bD^n} \nonumber \\
&=\ciptwo{|\bD^{n+1}-\bD^{n}|^2}{\frac{\kappa^2}{\dt^2\varepsilon}} \geq 0. \nonumber
\end{align}

The above analysis shows the energy dissipation \eqref{discrete energy decay} if 
\begin{equation}
\dt \leq \frac{I_1}{2I_2}. \nonumber
\end{equation}
Now we estimate the upper bound of $I_2$. The discrete Gauss's law shows that
\begin{equation}
2\kappa^2 \ciptwo{\nabla_h \cdot (\bD^{n+1}-\bD^{n})}{\phi^{n+1}-\phi^n} = \sum_{\ell=1}^M q^\ell \ciptwo{c^{\ell,n+1}-c^{\ell,n}}{\phi^{n+1}-\phi^n}, \nonumber
\end{equation}
where $\phi^n$ is the discrete electric potential. It's from the summation by parts that
\begin{equation}
LHS = 2\kappa^2 \ciptwo{\bD^{n+1}-\bD^{n}}{\frac{1}{\varepsilon}(\bD^{n+1}-\bD^{n})} \geq \frac{2\kappa^2}{\varepsilon_{\max}}\|\bD^{n+1}-\bD^{n}\|^2 \label{e proof 4}
\end{equation}
and
\begin{align}
RHS &= -\sum_{\ell=1}^M q^\ell \ciptwo{\dt f_e(L_h\dt) \bJ^{\ell,n} }{\frac{1}{\varepsilon}(\bD^{n+1}-\bD^{n})} \nonumber \\
&\leq \frac{\dt \lambda_{\max}(f_e(L_h\dt))}{\varepsilon_{\min}} \|\bD^{n+1}-\bD^{n}\| \sum_{\ell=1}^M |q^\ell| \|\bJ^{\ell,n}\| \nonumber \\
& = \frac{\dt f_e(\lambda\dt)}{\varepsilon_{\min}} \|\bD^{n+1}-\bD^{n}\| \sum_{\ell=1}^M |q^\ell| \|\bJ^{\ell,n}\|.  \label{e proof 5}
\end{align}
A combination of \eqref{e proof 4} and \eqref{e proof 5} shows that
\begin{equation}
\|\bD^{n+1}-\bD^{n}\| \leq \frac{ \dt \varepsilon_{\max} f_e(\lambda\dt)}{2\kappa^2\varepsilon_{\min}} \sum_{\ell=1}^M |q^\ell|\|\bJ^{\ell,n}\|\nonumber
\end{equation}
and
\begin{align}
I_2 &\leq \frac{\kappa^2}{\dt^2 \varepsilon_{\min}} \|\bD^{n+1}-\bD^{n}\|^2 \nonumber \\
& \leq \frac{\varepsilon_{\max}^2 f_e^2(\lambda\dt)\sum_{\ell=1}^M |q^\ell|^2}{4\varepsilon_{\min}^3\kappa^2} \sum_{\ell=1}^M \|\bJ^{\ell,n}\|^2. \label{e proof 6}
\end{align}
Combining \eqref{e proof 2} and \eqref{e proof 6}, we have
\begin{align}
\frac{I_1}{2I_2} & \geq \frac{2\varepsilon_{\min}^3\kappa^2 f_e(\lambda\dt + \frac{8\kappa \dt}{h^2}) }{\varepsilon_{\max}^2 f_e^2(\lambda\dt)\sum_{\ell=1}^M |q^\ell|^2}  \cdot
   \frac{\sum_{\ell=1}^M \ciptwo{ -\bJ^{\ell,n} }{\nabla_h \mu^{\ell,n,*}}}{\sum_{\ell=1}^M \|\bJ^{\ell,n}\|^2} \nonumber \\
& \geq C \min_{i,j,\ell} \left\{\frac{\mu_{i+1,j}^{\ell,n,*}-\mu_{i,j}^{\ell,n,*}}{-h \bJ_{i+\frac{1}{2}, j}^{\ell, n}}, \frac{\mu_{i,j+1}^{\ell,n,*}-\mu_{i,j}^{\ell,n,*}}{-h\bJ_{i,j+\frac{1}{2}}^{\ell,n}}\right\},  \nonumber
\end{align}
where
\begin{equation}
C=\frac{2\varepsilon_{\min}^3\kappa^2 f_e(\lambda\dt + \frac{8\kappa \dt}{h^2}) }{\varepsilon_{\max}^2 f_e^2(\lambda\dt)\sum_{\ell=1}^M |q^\ell|^2}. \nonumber
\end{equation}
Furthermore, based on the mean value theorem and the entropic mean \eqref{entropic mean}
\begin{align}
\frac{\mu_{i+1,j}^{\ell,n,*}-\mu_{i,j}^{\ell,n,*}}{-h J_{i+\frac{1}{2}, j}^{\ell, n}} &= \frac{\mu_{i+1,j}^{\ell,n,*}-\mu_{i,j}^{\ell,n,*}}{\kappa(e^{\mu_{i+1,j}^{\ell,n,*}}-e^{\mu_{i,j}^{\ell,n,*}})} \cdot \frac{e^{g^{\ell,n}_{i+1,j}}-e^{g^{\ell,n}_{i,j}}}{g^{\ell,n}_{i+1,j}-g^{\ell,n}_{i,j}} \nonumber \\
&=\frac{1}{\kappa e^{\theta \mu_{i+1,j}^{\ell,n,*} +(1-\theta)\mu_{i,j}^{\ell,n,*}}} \cdot \frac{e^{g^{\ell,n}_{i+1,j}}-e^{g^{\ell,n}_{i,j}}}{g^{\ell,n}_{i+1,j}-g^{\ell,n}_{i,j}} \nonumber \\
&=\frac{e^{-\theta g^{\ell,n}_{i+1,j}+(\theta-1) g^{\ell,n}_{i,j}}}{\kappa(c^{\ell,n+1}_{i+1,j})^\theta (c^{\ell,n+1}_{i,j})^{1-\theta} } \cdot \frac{e^{g^{\ell,n}_{i+1,j}}-e^{g^{\ell,n}_{i,j}}}{g^{\ell,n}_{i+1,j}-g^{\ell,n}_{i,j}} \nonumber \\
&\geq \frac{e^{-|dg^{\ell,n}_{i+1/2,j}|}}{\kappa c_{\max}}.  \nonumber
\end{align}
Similarly,
\begin{equation}
\frac{\mu_{i,j+1}^{\ell,n,*}-\mu_{i,j}^{\ell,n,*}}{-h J_{i, j+\frac{1}{2}}^{\ell, n}} \geq \frac{e^{-|dg^{\ell,n}_{i,j+1/2}|}}{\kappa c_{\max}}.   \nonumber
\end{equation}
Finally we have 
\begin{equation}
\frac{I_1}{2I_2} \geq \frac{C}{\kappa c_{\max}} \exp\left[ -\max_{i,j,\ell}\left( |dg^{\ell,n}_{i+1/2,j}|,|dg^{\ell,n}_{i,j+1/2}| \right) \right].\nonumber
\end{equation}
Therefore, a choice is that
\begin{equation}
\dt \leq \frac{2\varepsilon_{\min}^3\kappa f_e(\lambda\dt + \frac{8\kappa \dt}{h^2}) }{c_{\max}\varepsilon_{\max}^2 f_e^2(\lambda\dt)\sum_{\ell=1}^M |q^\ell|^2} \exp\left[ -\max_{i,j,\ell}\left( |dg^{\ell,n}_{i+1/2,j}|,|dg^{\ell,n}_{i,j+1/2}| \right) \right],\nonumber
\end{equation}
equivalently,
\begin{equation}
\frac{\dt f_e^2(\lambda\dt) }{f_e(\lambda\dt + \frac{8\kappa \dt}{h^2})} \leq \frac{2\varepsilon_{\min}^3\kappa}{c_{\max}\varepsilon_{\max}^2\sum_{\ell=1}^M |q^\ell|^2} \exp\left[ -\max_{i,j,\ell}\left( |dg^{\ell,n}_{i+1/2,j}|,|dg^{\ell,n}_{i,j+1/2}| \right) \right]. \nonumber
\end{equation}
To make it clearer, we approximate the upper bound of the left hand:
\begin{align}
\frac{\dt f_e^2(\lambda\dt) }{f_e(\lambda\dt + \frac{8\kappa \dt}{h^2})}& \leq \frac{\dt f_e(\lambda\dt) }{f_e(\lambda\dt + \frac{8\kappa \dt}{h^2})}  \nonumber \\
& = \left( \dt+\frac{8\kappa \dt}{\lambda h^2} \right) \frac{1-e^{-\lambda \dt}}{1-e^{-(\lambda \dt+\frac{8\kappa \dt}{h^2})}}  \nonumber \\
& \leq \dt+\frac{8\kappa \dt}{\lambda h^2}, \nonumber
\end{align}
where the bound $0<f_e(x)<1$, $x>0$ is applied. As a result, the energy dissipation is satisfied if
\begin{equation}
\dt+\frac{8\kappa \dt}{\lambda h^2} \leq \frac{2\varepsilon_{\min}^3\kappa}{c_{\max}\varepsilon_{\max}^2\sum_{\ell=1}^M |q^\ell|^2} \exp\left[ -\max_{i,j,\ell}\left( |dg^{\ell,n}_{i+1/2,j}|,|dg^{\ell,n}_{i,j+1/2}| \right) \right],\nonumber
\end{equation}
which completes the proof.
\end{proof}

\begin{remark}
    As a brief summary, the conditions of the Picard iteration and positivity are determined in the form of $\dt \leq O(h^2)$ by the matrix analysis. The coefficient can be relaxed by the stabilizer $\lambda$. However, we have not yet proved the unconditional positivity-preservation through stable terms and we may leave this to future work. If there is a more accurate estimate for the exponential matrix, the conditions may be relaxed. 
\end{remark}

\section{Numerical test} \label{section numerical test}
To assess the proposed numerical method in Algorithm \ref{algorithm:MANP}, simulations are conducted to verify the expected convergence order, positivity, energy dissipation and mass conservation at discrete level.
\subsection{Convergence rate}
In this example, we aim to test the convergence rate at the final time $T=0.1$ with a variable dielectric coefficient. The computational domain is set as $\Omega=[0,1]^2$. Since the exact solution is not explicit, we compute the Cauchy difference which is defined as
\begin{equation}
\delta_{\phi}:=\phi_{h_c}-\mathcal{I}_c^f\left(\phi_{h_f}\right), \nonumber
\end{equation}
where $h_f$ denotes the numerical solution on the fine mesh and $h_c$ is one on the coarse mesh. Here we set $h_c=2h_f$. $\mathcal{I}_c^f$ is a bilinear projection mapping the solution from fine grid to coarse grid. A smooth trigonometric initial date is taken as
\begin{equation}
c^1(t=0)=c^2(t=0)=0.4\sin (\pi x)\sin (\pi y)+0.5 \nonumber
\end{equation}
and relevant parameters are chosen as~$\Delta t=0.1h^2, \kappa=0.02$. The dielectric coefficient $\varepsilon$ is given as
\begin{equation}
\varepsilon=0.1\cos (\pi x)\cos (\pi y)+0.2.
\end{equation}
The $\ell^\infty$ and $\ell^2$ Cauchy differences and convergence orders of $c^1$ are presented in Table \ref{Convergence c1}. It is evident that as the mesh refines, the estimated orders gradually converge towards the theoretical value.

\begin{table}[h]
\center
\begin{tabular}{lllll}
\hline
                     & 16-32      & 32-64      & 64-128     & 128-256    \\ \hline
$\ell^\infty$ error  & 3.7068E-03 & 9.5586E-04 & 2.4085E-04 & 6.0330E-05 \\ 
$\ell^\infty$ order  &  -         & 1.9553     & 1.9887     & 1.9972     \\ 
$\ell^2$ error       & 1.9267E-04 & 4.8257E-04 & 1.2071E-04 & 3.0183E-05 \\ 
$\ell^2$ order       & -          & 1.9973     & 1.9992     & 1.9997     \\ \hline
\end{tabular} 
\caption{Convergence rates and Cauchy differences of $c^1$ in norm of both $\ell^2$ and $\ell^\infty$}\label{Convergence c1}
\end{table}


In addition, we test the convergence order using the modified equations with given exact solution. It is explicitly defined by adding source terms. The MANP-type equations are given as
\begin{align}
& \partial_t c^{\ell}=-\nabla \cdot \boldsymbol{J}^{\ell}, \ell=1,2, \nonumber \\
& \boldsymbol{J}^{\ell}=-\kappa\left(\nabla c^{\ell}-q^{\ell} c^{\ell} \boldsymbol{D} / \varepsilon+\boldsymbol{g}^{\ell}\right), \ell=1,2,  \nonumber \\
& \partial_t \boldsymbol{D}=\left(-\boldsymbol{J}^1+\boldsymbol{J}^2+\boldsymbol{S}\right) /\left(2 \kappa^2\right)+\boldsymbol{\Theta}, \nonumber \\
& \nabla \times \boldsymbol{D} / \varepsilon=0, \nonumber \\
& \nabla \cdot \boldsymbol{\Theta}=0, \nonumber
\end{align}
with the constant dielectric coefficient $\varepsilon=0.5$, valences $q^\ell=(-1)^{\ell+1}$, and $\kappa=1$. The source terms $\boldsymbol{g}^{\ell}$ and $\boldsymbol{S}$, as well as initial conditions, are determined by the following exact solution
\begin{align}
& c^{\ell}(x, y, t)=\pi^2 e^{-t} \cos (\pi x) \cos (\pi y) / 5+2, \ell=1,2,   \nonumber\\
& \boldsymbol{D}(x, y,t)=\left(\begin{array}{l}
\pi e^{-t} \sin (\pi x) \cos (\pi y) / 2 \\
\pi e^{-t} \cos (\pi x) \sin (\pi y) / 2
\end{array}\right). \nonumber
\end{align}
We take time step $\dt=0.1h^2$ and final time $T=0.1$. The result from the proposed method is shown in Table \ref{Convergence JCP 1}. This method also achieves theoretical convergence of first order in time and second order in space, further verifying the accuracy.


\begin{table}[h]
\center
\begin{tabular}{lllll}
\hline
   h   & $c^1$      & order   & $c^2$      & order       \\ \hline
 2/16  & 5.4137E-03 &  -      & 6.9079E-03 &  -          \\ 
 2/32  & 1.5763E-03 & 1.7800  & 2.0483E-03 & 1.7538      \\ 
 2/64  & 4.1024E-04 & 1.9420  & 5.6475E-04 & 1.8587      \\ 
 2/128 & 1.0830E-04 & 1.9214  & 1.4166E-04 & 1.9952      \\ \hline
\end{tabular} 
\caption{Convergence orders and errors of ETD scheme \eqref{ETD1} and \eqref{D* equation} in norm of $\ell^\infty$}
\label{Convergence JCP 1}
\end{table}

\subsection{Structure-preserving properties}\label{example 1}
In this part, a series of numerical tests are performed to verify the structure-preserving properties, including mass conservation, positivity and energy dissipation. The excess chemical potential $\mu^{\ell, \mathrm{cr}}$ is defined in \eqref{excess potential} and parameters are set as $\kappa=0.02$, $v^1=0.716^3$, $v^2=0.676^3$, $v^0=0.275^3$, $q^\ell=(-1)^{\ell+1}$ and $\chi=198.9437$. The stabilizer is chosen as $\lambda=2$. The space-dependent relative dielectric coefficient $\varepsilon(x)$ is determined by $\varepsilon_w$ and $\varepsilon_m$
\begin{equation}
\varepsilon(x, y)=\frac{\varepsilon_w-\varepsilon_m}{2}\left[\tanh \left(50 \sqrt{x^2+y^2}-25\right)+1\right]+\varepsilon_m. \nonumber
\end{equation}
The initial data is set as $c^{\ell,0}=0.1$ and the fixed charge distribution is given in polar coordinate
\begin{equation}
\rho^f=\mathbf{1}_{\left\{0.24 \leq r^2 \leq 0.26,0<\theta \leq \pi\right\}}-\mathbf{1}_{\left\{0.24 \leq r^2 \leq 0.26, \pi<\theta \leq 2 \pi\right\}}. \nonumber
\end{equation}
This configuration implies that the positive and negative charges of equal amount are uniformly distributed in the upper and lower sides, respectively, of a circular ring with a certain thickness. It is commonly referred to a Janus sphere. The computational domain is $[-1, 1]^2$, time step $\dt=10^{-4}$ and space step $h=2/128$. The following three examples are presented with different relative dielectric coefficients and parameters $\kappa$. All the parameters in the simulations are the same as those in the reference \cite{Qiao_NumericalMANP_2023} to test the reliability of the numerical methods.

\begin{example}\label{Exam 1}
A uniform dielectric coefficient is determined by $\varepsilon_w=\varepsilon_m=1$ and $\kappa=0.02$.
\end{example}
\begin{figure}[t!]
	\center
	\includegraphics[scale=0.75]{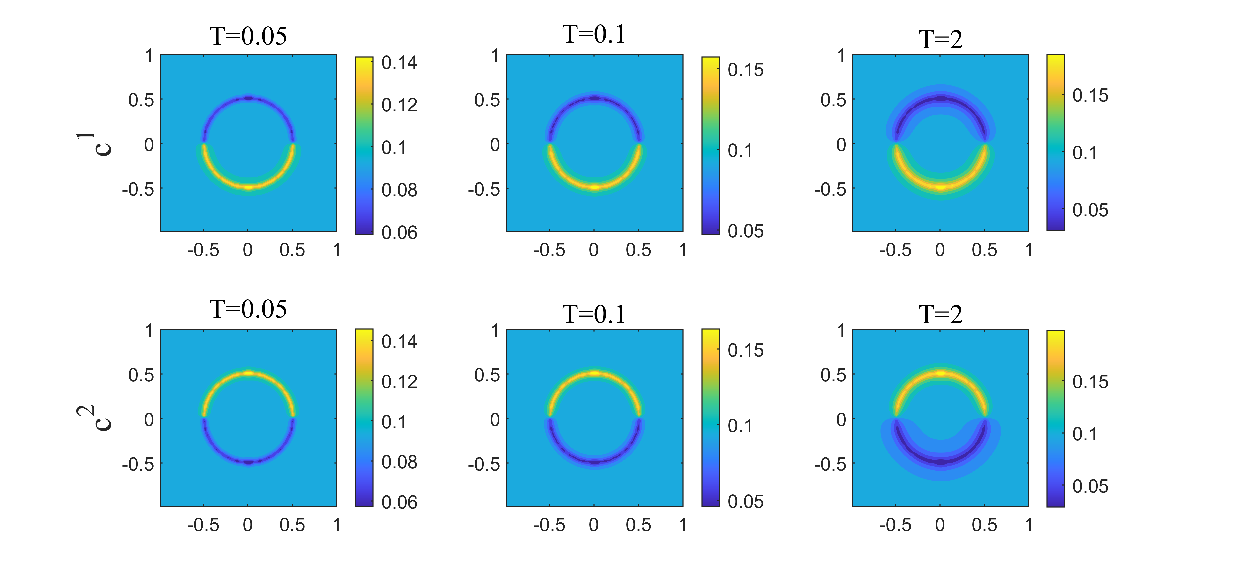}
	\caption{Slices of $c^1$ and $c^2$ at time $T=0.05,~ 0.1,~ 2$ in Example \ref{Exam 1}.} \label{ETD1_1_c1_c2}
\end{figure}
In the first example, a uniform dielectric coefficient is set for the numerical simulation where the Born solvation energy barrier does not have an impact on the results. Figure \ref{ETD1_1_c1_c2} shows the snapshots of $c^1$ and $c^2$ at selected times respectively. In the numerical simulations, it can be observed that under the influence of electrostatic interaction, ions with opposite signs gather at the locations of fixed charges, forming a narrow ring-like structure. The relatively small value of parameter $\kappa$ facilitates the formation of such structures.





\begin{example}\label{Exam 2}
A variable dielectric coefficient is determined by $\varepsilon_w=78$, $\varepsilon_m=1$ and $\kappa=0.02$.
\end{example}
In the second example, initial data and parameters are the same as Example \ref{Exam 1} except the variable dielectric coefficient. From Figure \ref{ETD1_78_c1_c2}, the electrostatic interaction generated by the fixed charges is greatly weakened compared to the previous Example \ref{Exam 1}. In this situation, the Born solvation energy barrier presents the dominance effect, which is caused by the significant difference in the relative dielectric coefficient. Under the driving force of convection, all free ions rapidly move from regions of low permittivity to regions of high permittivity. This movement leads to the ionic accumulation at the boundary separating two regions and then gradually spreads throughout the high dielectric coefficient region. 

Furthermore, the iteration steps of Picard iteration and curl-free relaxation in each time step are shown in Figure \ref{ETD1_78_Piccard}. 
In the initial steps, these two algorithms both require relatively more iterations to achieve convergence. After that, the solution can be efficiently obtained within very few steps, showing the efficiency of our proposed method. 


%

\begin{figure}[htp]
\center
\includegraphics[scale=0.75]{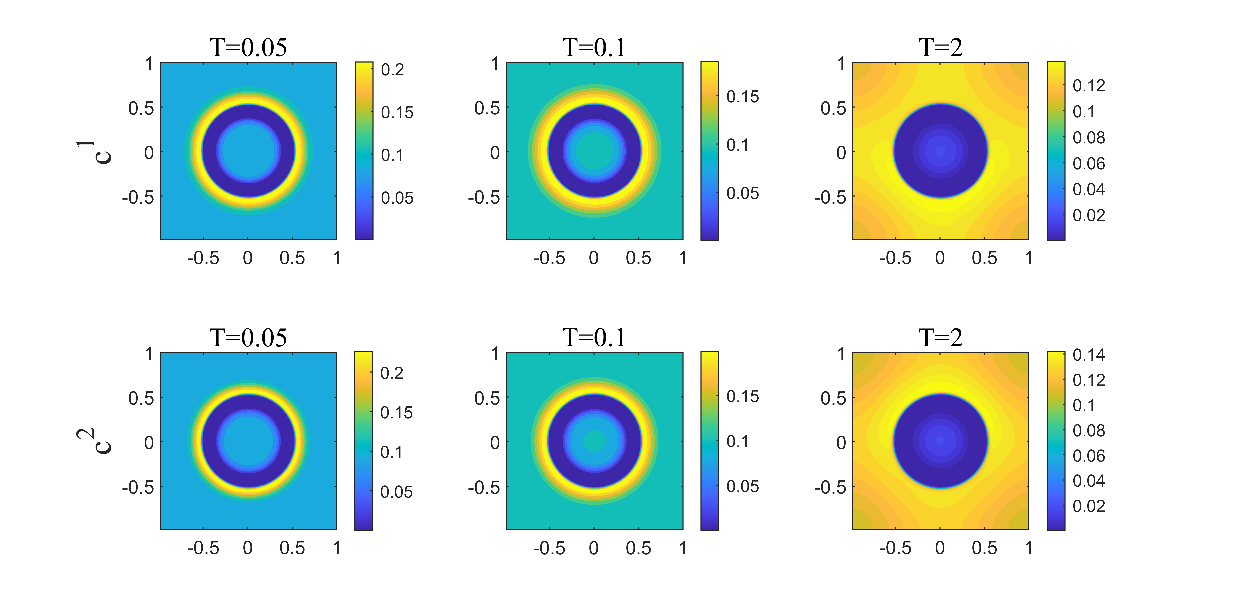}
\caption{Slices of $c^1$ and $c^2$ at time $T=0.05, ~0.1,~ 2$ in Example \ref{Exam 2}.} \label{ETD1_78_c1_c2}
\end{figure}

\begin{figure}[htp]
	\center
	\includegraphics[scale=0.45]{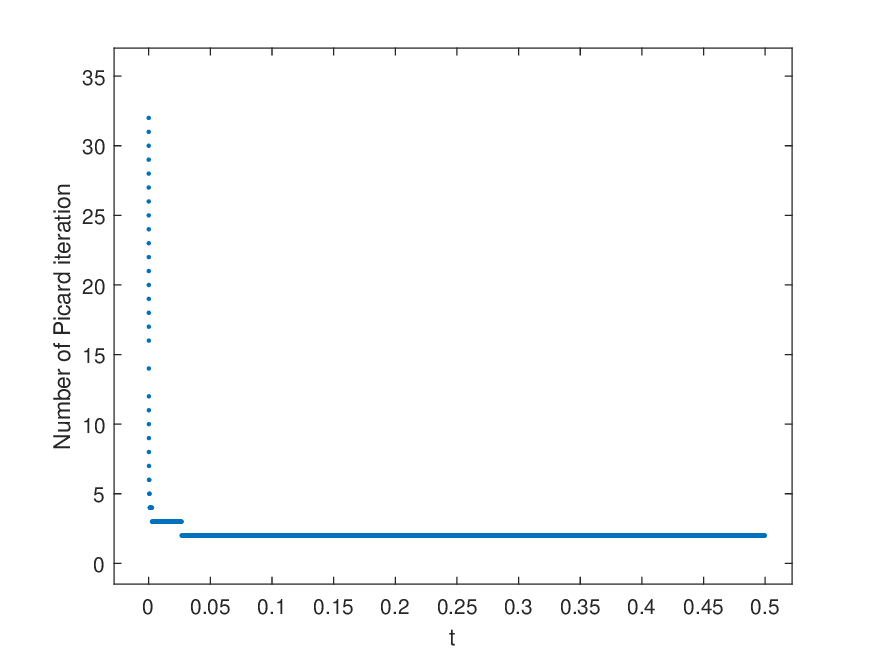}
	\includegraphics[scale=0.45]{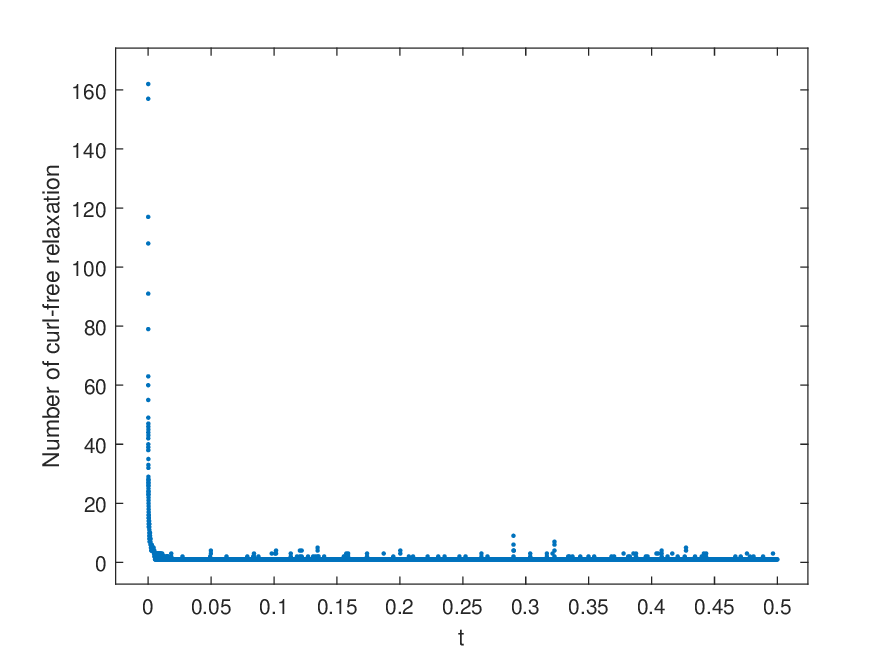}
	\caption{The number of Picard iteration and curl-free relaxation in Example \ref{Exam 2}.} 
\end{figure}


\begin{example}\label{Exam 3}
Variable dielectric coefficient is defined by $\varepsilon_w=78$, $\varepsilon_m=1$ and $\kappa=0.01$.
\end{example}
In the third example, to study the effect of Debye length, we calculate the dynamic under a smaller $\kappa=0.01$. As shown in Figure \ref{ETD1_78_c1_c2_k01}, both cations and anions once again gather at the interface and then diffuse, indicating that Born solvation interactions continue to play a dominant role. Additionally, the smaller $\kappa$ enhances convection and leads to narrower boundary layers. Compared Figure \ref{ETD1_78_c1_c2} with Figure \ref{ETD1_78_c1_c2_k01}, the different dynamics demonstrate that the electrostatic interaction becomes more pronounced with smaller $\kappa$. To summarize briefly, these numerical experiments show that our proposed algorithm is consistent with theoretical expectations in terms of accuracy and physical structure. It exhibits the same phenomena as described in the reference, and demonstrates good stability and reliability among various situations.

\begin{figure}[htp]
\center
\includegraphics[scale=0.75]{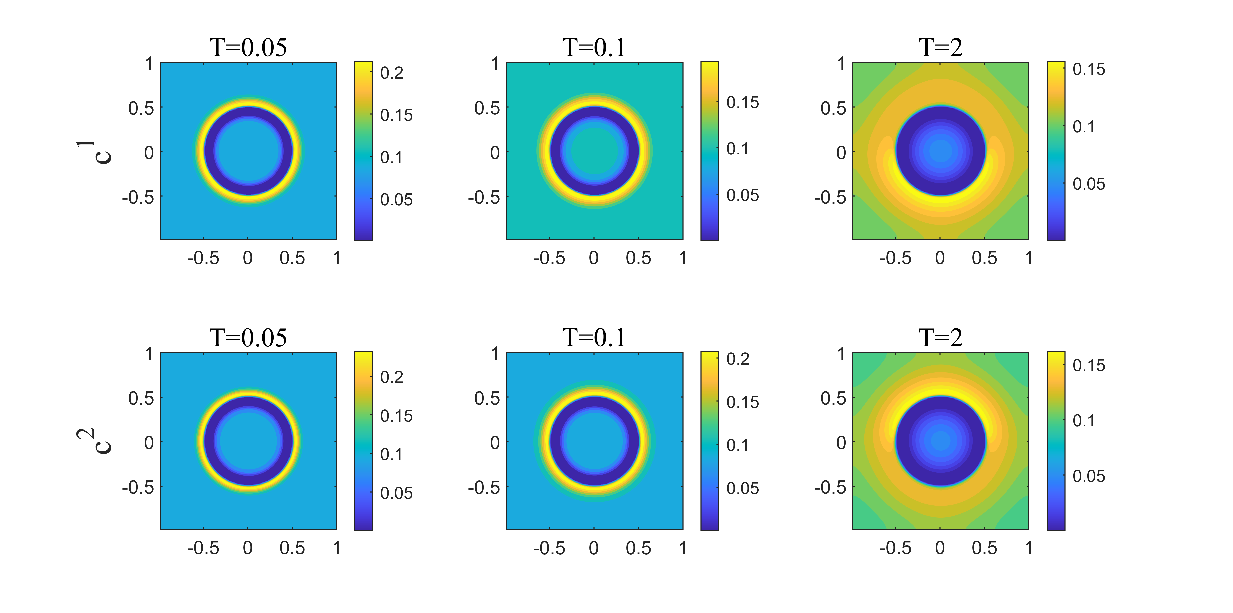}
\caption{Slices of $c^1$ and $c^2$ at time $T=0.05, ~0.1, ~2$ in Example \ref{Exam 3}. $\kappa=0.01$,$\varepsilon_w=78$ and $\varepsilon_m=1$.} \label{ETD1_78_c1_c2_k01}
\end{figure}


Furthermore, the structure-preserving properties, including positivity, mass concentration, and energy dissipation, are studied for the previous three examples. Figure \ref{Energy-mass-Cmin1} illustrates the evolution of normalized discrete energy and total mass of cation. Mass conservation and energy dissipation hold in all three numerical simulations. By comparison, one can observe that when significant disparity appears in the relative dielectric coefficient, the energy decreases more rapidly during the initial stage and then dissipates at a slower rate. Ultimately, it takes longer time to reach equilibrium than in the case where the relative dielectric coefficient is constant. 
Figure \ref{Energy-mass-Cmin2} shows the minimum of concentration $c_{\min}$ which is defined by
\begin{equation}
	c_{\min}:=\min_{i,j}\{c^0_{i,j},c^1_{i,j},c^2_{i,j}\}.
\end{equation} 
$c_{\min}$ gradually decreases to a pretty small value but always remains positive.

\begin{figure}[htp]
\center
{
\includegraphics[width=0.45\linewidth]{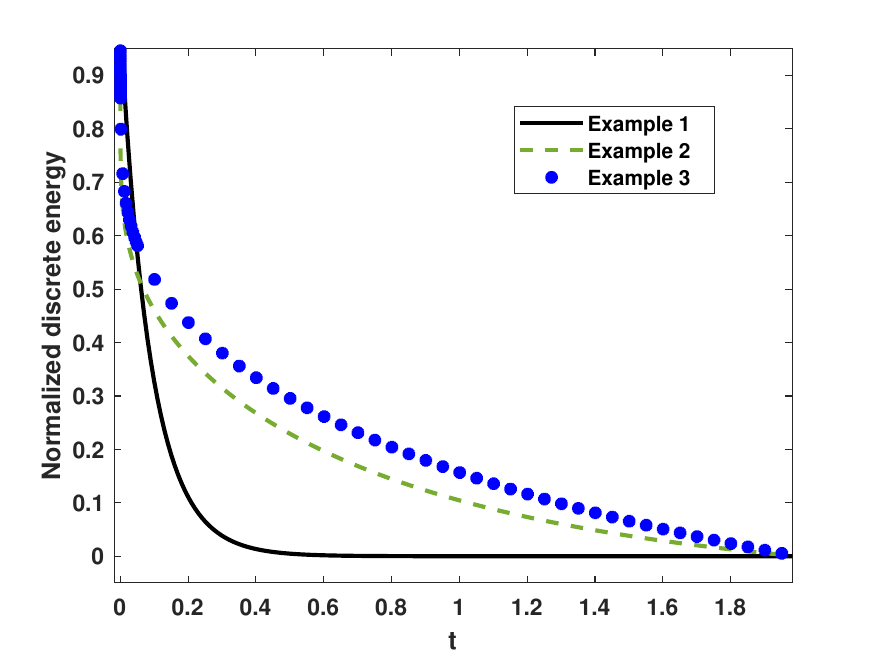}
}
{
\includegraphics[width=0.45\linewidth]{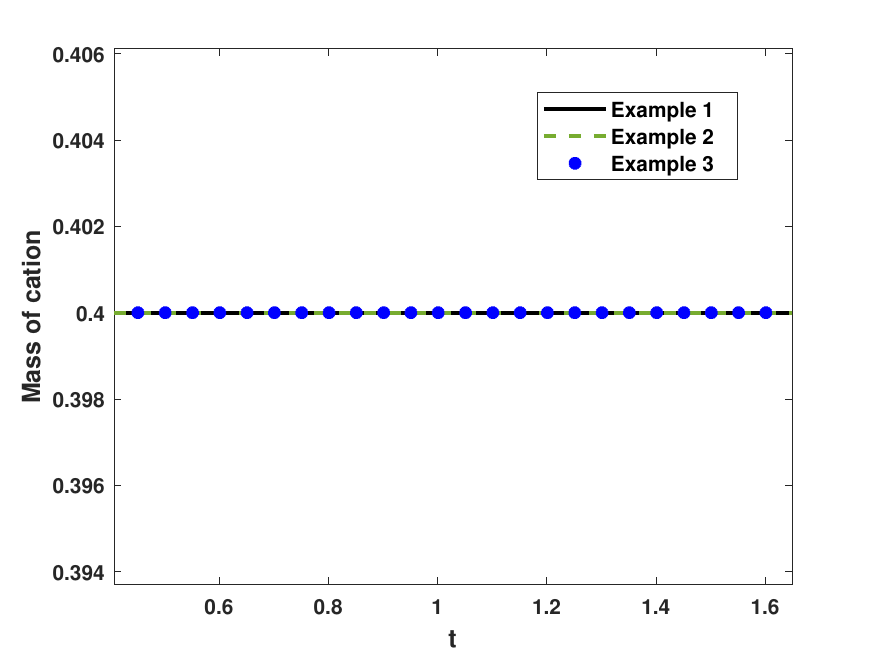}
}
\caption{ Numerical simulations for physical properties in three examples: (a) normalized discrete energies in three numerical simulations (b) total mass of cation.}
\label{Energy-mass-Cmin1}
\end{figure}

\begin{figure}[htp]
\center
\includegraphics[width=0.45\linewidth]{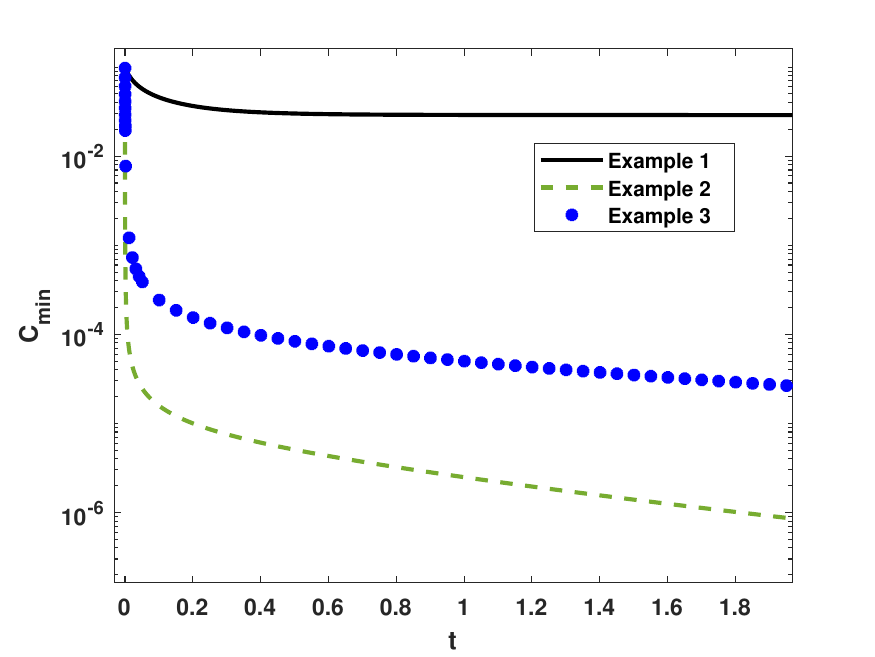}
\caption{Evolution of $c_{\min}$ in three examples.}
\label{Energy-mass-Cmin2}
\end{figure}

\section{Conclusion} \label{section conclusion}
In this work, we establish a decoupled implicit ETD method for the MANP model where steric effects and Born solvation interaction are involved to describe the complicated interactions of charged systems. An additional linear term is introduced to the Nernst-Planck equations and the reformulated equations are treated by the implicit ETD method which can be solved by the Picard iteration. Based on the skillful matrix analysis method and Slotboom transformation, a few conditions on temporal and spatial steps are proposed for the Picard iteration, positivity and energy dissipation respectively. We will try to weaken or remove these conditions in the future work. 
The flux in the Maxwell-Amp\`ere equation uses the same formulation as that of Nernst-Planck equations to fulfill Gauss's law.
To deal with the curl-free constraint, the dielectric displacement obtained from the Maxwell-Amp\`ere equation is further relaxed by a local algorithm of linear computational complexity. The relaxation provides a new idea for numerically solving Poisson's equation, especially with variable coefficient. Relevant numerical simulations demonstrate that our method can efficiently handle cases involving strong convection and maintain physical properties. The convergence analysis of the numerical method attaches to huge challenge and will be studied in our future work as well.

\begin{acknowledgements}
Z. Zhang is partially supported by the NSFC No.11871105 and No.12231003.
\end{acknowledgements}

\section*{Data Availibility}
Data will be made available on reasonable request.
\section*{Declarations}
\textbf{Conflict of interest}  The authors declare that they have no Conflict of interest concerning the publication of this manuscript.

\bibliographystyle{spmpsci}
\bibliography{ref}

%
%



\end{document}